\documentclass[twoside,11pt]{article}

\usepackage{jmlr2e}
\usepackage{graphicx}
\usepackage{psfrag}
\usepackage{caption}
\usepackage{subfig}
\usepackage{url}
\usepackage{epsfig}
\usepackage{natbib}
\usepackage{color}
\usepackage{amsmath,amssymb}
\usepackage{hyperref}

\newtheorem{assumption}[theorem]{Assumption}
\newtheorem{corrolary}[theorem]{Corollary}

\providecommand{\rset}[1]{\mathbb{R}^}

\providecommand{\norm}[1]{\lVert#1\rVert}

\ShortHeadings{Stochastic proximal point algorithms for constrained
convex optimization}{Patrascu and Necoara }

\firstpageno{1}

\begin{document}

\title{Nonasymptotic convergence of stochastic proximal point algorithms
for constrained convex optimization}

\author{\name Andrei Patrascu \email andrei.patrascu@fmi.unibuc.ro \\
       \addr Computer Science Department\\
             University of Bucharest\\
       \AND
       \name Ion Necoara \email ion.necoara@acse.pub.ro \\
       \addr Automatic Control and Systems Engineering Department\\
             University Politehnica of Bucharest
}

\editor{}

\maketitle

\begin{abstract}
A very popular approach for solving stochastic optimization problems
is the stochastic gradient descent method (SGD). Although the SGD
iteration is computationally cheap and the  practical performance of
this method may be satisfactory under certain circumstances, there
is recent evidence of its convergence difficulties and instability
for unappropriate parameters choice. To avoid these drawbacks
naturally introduced by the SGD scheme, the stochastic proximal
point algorithms have been recently considered in the literature. We
introduce a new variant of the stochastic proximal point method
(SPP) for solving stochastic convex  optimization problems subject
to (in)finite intersection of constraints satisfying a linear
regularity type condition. For the newly introduced SPP scheme we
prove new nonasymptotic convergence results. In particular, for
convex and Lipschitz continuous objective functions, we prove
nonasymptotic estimates for the rate of convergence in terms of the
expected value function gap of order
$\mathcal{O}\left(\frac{1}{k^{1/2}}\right)$, where $k$ is the
iteration counter. We also derive better  nonasymptotic bounds for
the rate of convergence in terms of  expected quadratic distance
from the iterates to the optimal solution for smooth strongly convex
objective functions, which in the best case is of order
$\mathcal{O}\left(\frac{1}{k}\right)$. Since these convergence rates
can be attained by our SPP algorithm only under some natural
restrictions on the stepsize, we also introduce a restarting variant
of SPP method that overcomes these difficulties and derive the
corresponding  nonasymptotic convergence rates. Numerical evidence
supports the effectiveness of our methods in real-world problems.
\end{abstract}

\begin{keywords}
Stochastic convex optimization, intersection of convex constraints, stochastic
proximal point method,  nonasymptotic convergence analysis.
\end{keywords}

\section{Introduction}
\noindent The randomness in  most of the practical optimization
applications  led the stochastic optimization field to become an
essential tool for many applied mathematics areas, such as machine
learning \cite{MouBac:11}, distributed optimization
\cite{NecNed:11}, sensor networks problems \cite{BlaHer:06}.  Since
the randomness usually enters the problem through the cost function
and/or the constraint set, in this paper we approach both randomness
sources and consider stochastic objective functions subject to
stochastic constraints. However, in the literature most of the time
the following unconstrained stochastic model has been considered:
\begin{equation}
\label{prob_prevwork}
\min\limits_{x \in \rset^n} F(x) = \left(\mathbb{E}[f(x;S)]\right).
\end{equation}
In the following subsections, we recall some popular numerical
optimization algorithms for solving the previous unconstrained
stochastic optimization model and set the context for our
contributions.

\subsection{Previous work}
A very popular approach for solving the unconstrained stochastic
problem  \eqref{prob_prevwork} is the stochastic gradient method
\cite{NemJud:09,MouBac:11,RosVil:14}. At each iteration $k$, the SGD
method independently samples a component function uniformly at
random $S_k$ and then takes a step along the gradient of the chosen
individual function, i.e.: \[ x^{k+1} = x^k - \mu_k \nabla
f(x^k;S_k), \] where $\mu_k$ is a positive stepsize. A particular
case of the continuous stochastic optimization model is the discrete
stochastic model, where the random variable $S$ is discrete and
thus, usually the objective function is given by the finite sum of
functional components. There exists a large amount of work in the
literature on deterministic and randomized algorithms for the
finite-sum optimization problems. Linear convergence results on a
restarted variant of SGD for finite-sum problems is given in
\cite{YanLin:16}. On the deterministic side, the incremental
gradient methods are the deterministic (cyclic) correspondent of the
SGD schemes and they were extensively analyzed in \cite{Ber:11}.
Another efficient class of algorithms for finite-sums are based on
the common idea of updating the current iterate along the aggregated
gradient step: e.g. incremental aggregated gradient (IAG)
\cite{VanGur:16} or SAGA algorithm \cite{DefBac:14}. There is a
recent nonasymptotic convergence analysis of SGD provided in
\cite{MouBac:11}, under various differentiability assumptions on the
objective function. While the SGD scheme is the method of  choice in
practice for many machine learning applications due to its superior
empirical performance, the theoretical estimates obtained in
\cite{MouBac:11} highlights several difficulties regarding its
practical limitations and robustness. For example, the stepsize is
highly constrained to small values by an exponential term from the
convergence rate which could be catastrophically increased by
uncontrolled variations of the stepsize. More precisely, the
convergence rates of SGD with decreasing stepsize
$\mu_k=\frac{\mu_0}{k}$, given for the quadratic mean
$\mathbb{E}[\norm{x^k-x^*}^2]$, where $x^*$ is an optimal solution
of \eqref{prob_prevwork}, contains certain exponential terms in the
initial stepsize of the following form \cite{MouBac:11}:
\begin{equation*}
\mathbb{E}[\norm{x^k-x^*}^2] \le \frac{C_1 e^{C_2\mu^2_0}}{k^{\alpha
\mu_0}} + \mathcal{O}\left(\frac{1}{k}\right),
\end{equation*}
for  $\mu_0 > 2/\alpha$, where  $C_1, C_2$ and $\alpha$ are some
positive constants. It is clear from previous convergence rate that
$\mathbb{E}[\norm{x^k-x^*}^2]$ can grow exponentially until the
stepsizes become sufficiently small, a behavior which can be also
observed in practical simulations.

\vspace{5pt}

\noindent Since these drawbacks are naturally introduced by the SGD
scheme, other modifications have been considered for avoiding these
aspects. One of them is the stochastic proximal point (SPP)
algorithm for solving the unconstrained stochastic problem
\eqref{prob_prevwork}   having the following iteration
\cite{RyuBoy:16}, \cite{TouTra:16}, \cite{Bia:16}:
\[ x^{k+1} = \arg \min_{z \in \rset^n} \left[ f(z; S_k) + \frac{1}{2 \mu_k} \|z - x^k\|^2 \right].   \]
Note that the SGD is the particular SPP method applied  to the
linearization of $f(z;S_k)$ in $x^k$, that is to the linear function
$l_f(z;x^k,S_k) = f(x^k;S_k) + \langle  \nabla f(x^k;S_k), z - x^k
\rangle$. Of course, when $f$ has an easily computable proximal
operator, it is natural to use $f$ instead of its linearization
$l_f$.  In \cite{RyuBoy:16}, the SPP algorithm has been applied to
problems with the objective function having Lipschitz continuous
gradient and the following \textit{restricted strong convexity}
property: \begin{align} \label{rsc} f(x;S) \ge f(y;S) + \langle
\nabla f(y;S), x-y\rangle + \frac{1}{2} \langle M_S(x-y),x-y \rangle
\quad \forall x,y \in \rset^n,
\end{align} for some matrix $M_S \succeq 0$, satisfying $\lambda := \lambda_{\min}(\mathbb{E}[M_S]) >0$. The SPP
algorithm has been analyzed in \cite{RyuBoy:16} under the above
assumptions  and the central results were the asymptotic global
convergence estimates of SPP with decreasing stepsize
$\mu_k=\frac{\mu_0}{k}$ and a nonasymptotic analysis for the SPP with
constant stepsize. In particular, it has been proved that  SPP
converges linearly to a noise-dominated region around the optimal
solution. Moreover,  the following \textit{asymptotic} convergence
rates in the quadratic mean (i.e. for a \textit{sufficiently large}
$k$) have been given:
\[ \mathbb{E}[\norm{x^k -x^*}^2] \le \left( \frac{1}{e} \right)^{\mu_0 \lambda \ln{(k+1)}}C_1 +\begin{cases}
\hspace{0.3cm} \frac{C_2}{\mu_0 \lambda - 1}\frac{1}{k} & \text{if} \quad \mu_0 \lambda > 1 \\
\hspace{0.3cm} \frac{C_2\ln(k)}{k} & \text{if} \quad \mu_0 \lambda = 1 \\
\frac{C_2}{(1 - \mu_0 \lambda) k^{\mu_0 \lambda}} & \text{if} \quad  \mu_0 \lambda<1,
\end{cases}
 \]
where $C_1$ and $C_2$ are some positive constants. With the
essential difference that no exponential terms in $\mu_0$ are
encountered, these rates of convergence have similar orders with
those for the classical SGD method with variable stepsize. Although
in this paper we make similar assumptions on the objective function,
we complement and extend the previous results. In particular, we
provide a \textit{nonasymptotic} convergence analysis of the
stochastic proximal point method  for a more general stepsize
$\mu_k=\frac{\mu_0}{k^\gamma}$, with $\gamma >0$, and for  \textit{constrained}
problems. Moreover, the Moreau smoothing framework used in the present
paper leads to more elegant and intuitive proofs. Another paper
related to the SPP algorithm is \cite{TouTra:16}, where the
considered stochastic model involves minimization of the expectation
of random particular components $f(x;S)$ defined by the composition
of a smooth function and a linear operator, i.e.: \[  f(x;S) = f(A^T_S x).
\]  Moreover, the objective function $F(x)= \mathbb{E}[f(A^T_S x)]$ needs to
satisfy $\lambda_{\min}\left(\nabla^2 F(x)\right) \ge \lambda
> 0$ for all $x \in \rset^n$. The nonasymptotic convergence of the
SPP with decreasing stepsize $\mu_k=\frac{\mu_0}{k^{\gamma}}$, with
$\gamma \in (1/2, 1]$, has been analyzed in the quadratic mean and
the following convergence rate has been derived in \cite{TouTra:16}:
\[ \mathbb{E}[\norm{x^k-x^*}^2] \le C \left(\frac{1}{1+\lambda \mu_0 \alpha}\right)^{k^{1-\gamma}}
+  \mathcal{O}\left(\frac{1}{k^{\gamma}}\right), \] where $C$ and
$\alpha$ are some positive constants. However, the analysis in
\cite{TouTra:16} cannot be trivially extended to the general convex
objective functions and complicated constraints, since for the proofs  it is essential that each
component of the objective function has the form $f(A^T_S x)$. In our paper we consider
general  convex objective functions which lack the previously
discussed structure and also (in)finite number of convex  constraints. Further, in
\cite{Bia:16} a  general asymptotic convergence analysis of several
variants of SPP scheme within operator theory settings has been
provided, under mild (strong) convexity assumptions. A particular optimization
model analyzed in \cite{Bia:16}, related to our paper, is:
\[ \min_{x} \; f(x) \quad \text{s.t.} \;\; x \in \cap_{i=1}^m X_i,
\] for which the following SPP type algorithm has been derived:
$$x^{k+1} =
\begin{cases}
\arg \min_{z \in \rset^n} \left[ f(z) + \frac{1}{2 \mu_k} \|z -
x^k\|^2 \right]
& \text{if} \; S_{k} = 0\\
 [x^k]_{X_{S_k}} & \text{otherwise},
\end{cases}
$$
where $S_k$ is a random variable on $\Omega = \{0, 1, \cdots, m\}$
with probability distribution $\mathbb{P}$. Although this scheme is
very similar with the  SPP algorithm, only the almost sure
asymptotic convergence has been provided in \cite{Bia:16}. Other
stochastic proximal (gradient) schemes together with their
theoretical guarantees are studied in several recent papers as we
further exemplify. In \cite{AtcFor:14} a perturbed proximal gradient
method is considered for solving composite optimization problems,
where the gradient is intractable and approximated by Monte Carlo
methods. Conditions on the stepsize and the Monte Carlo batch size
are derived under which the convergence is guaranteed. Two classes
of stochastic approximation strategies (stochastic iterative
Tikhonov regularization and the stochastic iterative proximal point)
are analyzed in \cite{KosNed:13} for monotone stochastic variational
inequalities and almost sure convergence results are presented. A
new  stochastic optimization method is analyzed  in \cite{YurVu:16}
for the minimization of the sum of three convex functions, one of
which has Lipschitz continuous gradient and satisfies a restricted
strong convexity condition. New convergence results are provided in
\cite{RosVil:17} for the stochastic proximal gradient algorithm
suitable for solving a large class of convex composite optimization
problems. The authors derive $\mathcal{O}\left( \frac{1}{k} \right)$
nonasymptotic bounds in expectation in the strongly convex case, as
well as almost sure convergence results under weaker assumptions. In
\cite{ComPes:16} the asymptotic behavior of a stochastic
forward-backward splitting algorithm for finding a zero of the sum
of a maximally monotone set-valued operator and a cocoercive
operator in Hilbert spaces is investigated. Weak and strong almost
sure convergence properties of the iterates are established under
mild conditions on the underlying stochastic processes.  In
\cite{Xu:11} the author presents a finite sample analysis for the
averaged SGD, which shows that it usually takes a huge number of
samples for averaged SGD to reach its asymptotic region, for
improperly chosen learning rate (stepsize). Moreover, he proposes a
simple  way to properly set the learning rate so that it takes a
reasonable amount of data for averaged SGD to reach its asymptotic
region. In \cite{NiuRec:11} the authors show through a novel
theoretical analysis that SGD can be implemented in a parallel
fashion without any locking. Moreover,  for sparse optimization
problems (meaning that most gradient updates only modify small parts
of the decision variable) the developed scheme achieves a nearly
optimal rate of convergence. A regularized stochastic version of the
BFGS method is proposed in \cite{MokRib:14} to solve convex
optimization problems. Convergence analysis shows that lower and
upper bounds on the Hessian eigenvalues of the sample functions are
sufficient to guarantee convergence of order $\mathcal{O}\left(
\frac{1}{k} \right)$. A comprehensive survey on  optimization
algorithms for machine learning problems is given recently in
\cite{BotCur:16}. Based on their experience, the authors present
theoretical results on a straightforward, yet versatile SGD
algorithm, discuss its practical behavior, and highlight
opportunities for designing new algorithms with improved
performance.


\subsection{Contributions}
\noindent In this paper we consider both randomness sources (i.e.
objective function and constraints) and thus our  problem of
interest involves stochastic objective functions subject to
(in)finite intersection of constraints.   As we previously observed,
given the clear superior features of SPP algorithm over the
classical SGD scheme, we also consider an  SPP scheme for solving
our problem of interest. The main contributions of this paper are:

\vspace{0.1cm}

\noindent \textit{More general stochastic optimization model and a
new stochastic proximal point algorithm}: While most of the existing
papers from the stochastic optimization literature consider convex
models without constraints or simple constraints, that is the
projection onto the feasible set is easy,  in this paper we consider
stochastic convex optimization problems subject to (in)finite
intersection of constraints satisfying a linear regularity type
condition. It turns out that many practical applications, including
those from  machine learning, fits into this framework: e.g.
regression problems, finite sum problems, portfolio optimization
problems, convex feasibility problems, etc. For this general
stochastic optimization model we introduce a new stochastic proximal
point (SPP) algorithm. It is worth to mention that although the
analysis of an SPP method for stochastic models with complicated
constraints is non-trivial, our framework allows us to deal with
even  an infinite number of constraints.

\noindent \textit{New nonasymptotic convergence results for the SPP
method}: For the newly introduced SPP scheme we prove new
nonasymptotic convergence results. In particular, for convex and
Lipschitz continuous objective functions, we prove nonasymptotic
estimates for the rate of convergence of the SPP scheme in terms of
the expected value function gap and feasibility violation of order
$\mathcal{O}\left(\frac{1}{k^{1/2}}\right)$, where $k$ is the
iteration counter. We also derive better nonasymptotic bounds for
the rate of convergence of SPP scheme with decreasing stepsize
$\mu_k = \frac{\mu_0}{k^{\gamma}}$, with  $\gamma \in (0,1]$, for
smooth $\sigma_{f,S}-$strongly convex objective functions. For this case the
convergence rates are given in terms of expected quadratic distance
from the iterates to the optimal solution and are of order:
\[ \mathbb{E}[\norm{x^k-x^*}^2] \le C\left(\mathbb{E}\left[\frac{1}{1+\sigma_{f,S} \mu_0 }\right]\right)^{k^{1-\gamma}}
+  \mathcal{O}\left(\frac{1}{k^{\gamma}}\right). \]
Note  that the derived
rates of convergence do not contain any exponential term in $\mu_0$,
as is the case of the SGD scheme.

\noindent \textit{Restarted variant of SPP algorithm and the
corresponding convergence analysis}: Since the best complexity of
our basic SPP scheme can be attained  only under some natural
restrictions on the initial stepsize $\mu_0$,  we also introduce a
restarting stochastic proximal point  algorithm that overcomes these
difficulties. The main advantage of this restarted variant of SPP
algorithm is that it is parameter-free and  thus it is easily
implementable in practice. Under strong convexity and smoothness
assumptions, for $\gamma > 0$ and epoch counter $t$, the restarting
SPP scheme with the constant stepsize (per epoch)
$\frac{1}{t^{\gamma}}$ provides a nonasymptotic complexity of order
$\mathcal{O}\left(\frac{1}{\epsilon^{1 + \frac{1}{\gamma}}}
\right)$.

\vspace{5pt}

\noindent \textbf{Paper outline}. The paper is organized as follows.
In Section 2 the problem of interest is formulated and analyzed.
Further in Section 3, a new stochastic proximal point algorithm is
introduced and its relations with the previous work are highlighted.
We provide in Section 4 the first main result of this paper
regarding the nonasymptotic convergence of SPP in the convex case.
Further, stronger convergence results are presented in Section 5 for
smooth strongly convex objective functions. In order to improve the
convergence of the simple SPP scheme, in Section 6 we introduce a
restarted variant of SPP algorithm. Lastly, in Section 7 we provide
some preliminary numerical simulations to highlight the empirical
performance of our schemes.

\vspace{5pt}

\noindent \textbf{Notations}. We consider the space $\rset^n$
composed  by column vectors. For $x,y \in \rset^n$ denote the scalar
product  $\langle x,y \rangle = x^T y$ and Euclidean norm by
$\|x\|=\sqrt{x^T x}$. The projection operator onto the nonempty
closed convex set $X$ is denoted by $[\cdot]_X$ and the distance
from a given $x$ to set $X$ is denoted by $\text{dist}_X(x)$. We
also define the function $\varphi_{\alpha}: (0, \infty) \to
\rset^{}$:
\[  \varphi_{\alpha}(x) =
\begin{cases}
 (x^{\alpha} - 1)/\alpha, & \text{if} \; \alpha \neq 0 \\
 \log(x), & \text{if} \; \alpha = 0.
\end{cases}
\]


\section{Problem formulation}
\noindent In many machine learning applications randomness usually
enters the problem through the cost function and/or the constraint
set. Minimization of problems having complicating constraints  can
be very challenging. This is usually alleviated by approximating the
feasible set by an (in)finite intersection of simple  sets
\cite{NecRic:16,Ned:11}. Therefore, in this paper we tackle the
following stochastic convex constrained optimization problem:
\begin{equation}
\label{convexfeas_random}
\begin{split}
F^* = & \min_{x \in \rset^n}  F(x) \;\; \left(: = \mathbb{E}[f(x;S)]\right) \\
& \text{s.t.}  \quad x \in X \;\; \left(: = \cap_{S \in \Omega} X_S\right),
\end{split}
\end{equation}
where $f(\cdot;S): \rset^n \to \rset^{}$ are convex  functions with
full  domain $\text{dom}f=\rset^n$, $X_S$ are nonempty closed convex
sets, and $S$ is a random variable with its associated probability
space $(\Omega, \mathbb{P})$.
Notice that this formulation allows us to include (in)finite number of constraints.
We denote the set of optimal solutions with $X^*$ and $x^*$ any
optimal point for \eqref{convexfeas_random}.
For the optimization problem
\eqref{convexfeas_random}  we make the following assumptions.
\begin{assumption}
\label{assump_fsmooth} For any $S \in \Omega$, the function
$f(\cdot;S)$ is proper, closed, convex and Lipschitz continuous,
that is there exists $L_{f,S} > 0$ such that
\begin{equation*}
|f(x;S) - f(y;S)| \le L_{f,S} \norm{x-y} \qquad \forall x,y \in
\rset^n.
\end{equation*}
\end{assumption}

\noindent Notice that Assumption \ref{assump_fsmooth} implies that
any subgradient $g_f(x;S) \in \partial f(x;S)$ is bounded, that is
$\norm{g_f(x;S)} \le L_{f,S}$ for all $x \in \rset^n$  and $S \in
\Omega$. For the sets we assume:

\begin{assumption}
\label{assump_feasset}
Given $S \in \Omega$, the following two properties hold:\\
\noindent $(i)$ $X_S$ are simple convex sets (i.e. projections onto these sets are easy).

\noindent $(ii)$ There exists $\kappa > 0$ such that the feasible
set $X$ satisfies linear regularity:
\begin{equation*}
\text{dist}_X^2(x) \; \le \kappa \;
\mathbb{E}[\text{dist}_{X_S}^2(x)] \qquad \forall x \in \rset^n.
\end{equation*}
\end{assumption}

\noindent Assumption \ref{assump_feasset} $(ii)$  is known in the
literature as the \textit{linear regularity property} and it is
essential for proving  linear convergence for (alternating)
projection algorithms, see \cite{NecRic:16,Ned:11}. For example,
when $X_S$ are  hyperplanes, halfspaces or  convex sets with
nonempty interior, then linear regularity property holds. The linear
regularity property is related to quadratic functional growth
condition introduced for smooth  convex functions in
\cite{NecNes:16}.   In \cite{NecNes:16} it has also been proved that
several  first order methods converge linearly  under functional
growth condition and smoothness of the objective function. Notice
that this general model \eqref{convexfeas_random} covers interesting
particular cases which we discuss below.

\subsection{Convex feasibility problem}
Let us consider the following objective function
and constraints:
\[ f(x;S) := \frac{\lambda}{2}\|x\|^2 \;\;\;
\forall S \in \Omega  \quad \text{and} \quad
X=\cap_{S \in \Omega} X_S, \] where $\lambda >0$. Then, we obtain
the least norm convex feasibility problem:
\[  \min_{x \in \rset^n} \frac{\lambda}{2}\|x\|^2 \quad \text{s.t.} \;\; x \in \cap_{S \in \Omega} X_S. \]
We can also consider another reformulation of the least norm  convex
feasibility problem:
\[ f(x;S) := \frac{\lambda_S}{2}\|x\|^2 + \mathbb{I}_{X_S}(x) \quad \forall S \in \Omega, \]
where $\lambda_S \geq 0$ and  $\mathbb{E}[\lambda_S]=\lambda$. Then,
this leads to the stochastic optimization model:
\[  \min_{x \in \rset^n}  \; \mathbb{E}\left[
\frac{\lambda_S}{2}\|x\|^2 + \mathbb{I}_{X_S}(x) \right]. \] Finding
a point in the intersection of a collection of closed convex sets
represents a modeling paradigm for solving important applications
such as data compression, neural networks  and  adaptive filtering,
see  \cite{CenChe:12} for a complete list.

\subsection{Regression problem} Let us consider the matrix
$A \in \rset^{m \times n}$. For any $S \in \Omega \subseteq \rset^{}$, let us define:
\[ f(x;S):=  \ell(A_S^T x), \] where $\ell$ is some loss function.   This
results in the following constrained optimization model:
\begin{align*}
\min\limits_{x \in \rset^n} \quad \mathbb{E}[\ell(A_S^T x)] \quad
\text{s.t.} \;\; x \in \cap_{S \in \Omega} X_S.
\end{align*}
Many learning problems can be modeled into this  form, see e.g.
\cite{TouTra:16}. This type of  optimization model has been also
considered in \cite{Bia:16, RosVil:14}.

\subsection{Finite sum problem} Let $\Omega = \{1, \cdots, m\}$ and
$\mathbb{P}$ be the uniform discrete probability distribution on
$\Omega$. Further, we consider convex functions $f(x;i) =
\ell_i(x)$. Then, the following constrained  finite sum problem is
recovered:
\begin{align*}
\min\limits_{x \in \rset^n} \quad \frac{1}{m} \sum\limits_{i=1}^m \ell_i(x) \quad \text{s.t.} \;\; x \in \cap_{i =1}^m X_i.
\end{align*}
This constrained optimization model appears often  in statistics and
machine learning applications, where the functions $\ell_i(\cdot)$
typically represent loss functions associated to a given estimator
and the feasible set comes from physical constraints, see e.g.
\cite{DefBac:14, VanGur:16, YurVu:16}. It is also a particular
problem of a more general optimization model considered in
\cite{Bia:16}.


\section{Stochastic Proximal Point algorithm}
In this section we propose solving the optimization problem
\eqref{convexfeas_random} through stochastic proximal point type
algorithms.  It has been proved in \cite{NecRic:16} that the
optimization problem \eqref{convexfeas_random} can be equivalently
reformulated under Assumption \ref{assump_feasset} into the
following stochastic optimization problem:
\begin{align}
\label{convexfeas_random_ref}
\min\limits_{x \in \rset^n} \;\; & F(x) \; \left(:= \mathbb{E} \left[ f(x;S) + \mathbb{I}_{X_S}(x) \right] \right).
\end{align}
Since each component of the stochastic objective is nonsmooth, a first possible
approach is to apply  stochastic subgradient methods
\cite{DucSin:09,MouBac:11}, which would yield simple algorithms, but having
usually a relatively slow sublinear convergence rate. Therefore, for
more robustness, one can deal with the nonsmoothness through the
Moreau smoothing framework. However, there are multiple potential
approaches in this direction. For a given smoothing parameter
$\mu>0$, we can smooth each functional component and the associated
indicator function together to obtain the following smooth approximation
for the nonsmooth convex function $f(\cdot;S) + \mathbb{I}_{X_S}$:
\[ \bar f_{\mu}(x;S) := \min\limits_{z \in \rset^n} \; f(z;S) +  \mathbb{I}_{X_S}(z) + \frac{1}{2\mu}\norm{z-x}^2. \]
Let us  denote the corresponding prox operator by
$\bar{z}_{\mu}(x;S) = \arg\min\limits_{z \in \rset^n} f(z;S) +
\mathbb{I}_{X_S}(z) + \frac{1}{2\mu} \norm{z - x}^2$.  It is known
that any Moreau approximation $\bar f_{\mu}(\cdot;S)$ is
differentiable having the gradient  $\nabla \bar f_{\mu}(x;S) =
\frac{1}{\mu}(x - \bar{z}_{\mu}(x;S))$. Moreover,  the gradient is
Lipschitz continuous with constants bounded by $\frac{1}{\mu}$.
Then, instead of solving nonsmooth problem
\eqref{convexfeas_random_ref} we can consider solving the smooth
approximation:
\[ \min\limits_{x \in \rset^n} \;\;  \bar F_{\mu}(x) \;\; \left(:= \mathbb{E}[\bar f_{\mu}(x;S)] \right).   \]
Notice that we can easily apply the classical SGD strategy  to the
newly created smooth objective function, which results in the
following iteration:
\begin{align*}
x^{k+1} & = x^k - \mu_k \nabla \bar f_{\mu_k}(x^k;S_k) = \bar{z}_{\mu_k}(x^k;S_k)\\
& = \arg\min\limits_{z \in \rset^n} \; f(z;S_k) +
\mathbb{I}_{X_{S_k}}(z) +  \frac{1}{2\mu_k}\norm{z-x^k}^2.
\end{align*}
However, the nonasymptotic analysis technique considered in our
paper  encounters  difficulties with this variant of the algorithm.
The main difficulty consists in proving the bound $\norm{\nabla \bar
f_{\mu}(x;S)} \le \norm{g_{f(\cdot;S) + \mathbb{I}_{X_S}}(x)}$ for
all $x \in \rset^n$, where $g_{f(\cdot;S) + \mathbb{I}_{X_S}}(x) \in
\partial (f(\cdot;S) + \mathbb{I}_{X_S})(x)$. We believe that such a
bound is essential in our convergence analysis and we leave for
future work the analysis of this iterative scheme. Therefore, we
considered a second approach based on a smooth Moreau approximation
only for  the functional component $f(\cdot;S)$ and keeping the
indicator function $\mathbb{I}_{X_S}$ in its original form, that is:
\[ f_{\mu} (x;S) := \min\limits_{z \in \rset^n}  \;  f(z;S) + \frac{1}{2\mu} \norm{z - x}^2 \]
for some  smoothing parameter $\mu>0$. Then, instead of solving nonsmooth problem \eqref{convexfeas_random_ref},
we solve the following composite approximation:
\begin{align}
\label{smooth_linearoptprob}
\min\limits_{x \in \rset^n} \;\;  F_{\mu}(x) \;\; \left(:= \mathbb{E}[ f_{\mu}(x;S) + \mathbb{I}_{X_{S}}(x)]\right).
\end{align}
Let us  denote the corresponding prox operator by:
\begin{align*}
z_{\mu}(x;S) &= \arg\min\limits_{z \in \rset^n} f(z;S) + \frac{1}{2\mu} \norm{z - x}^2.
\end{align*}
Further, on the stochastic composite approximation
\eqref{smooth_linearoptprob} we can apply the stochastic projected
gradient method,  which leads to a stochastic proximal point like
scheme for solving the original problem \eqref{convexfeas_random}:
\begin{center}
\framebox{
\parbox{11 cm}{
\begin{center}
\textbf{ Algorithm {\bf SPP} ($x_0, \{\mu_k\}_{k\ge 0}$) }
\end{center}
{ For $k\geq 1$ compute:}\\
1. Choose randomly $S_k \in \Omega$ w.r.t. probability distribution $\mathbb{P}$\\
2. Update: $y^{k} = z_{\mu_k}(x^{k};S_k)$ and $x^{k+1} = [y^{k}]_{X_{S_k}}$\\
}}
\end{center}
where $x^0 \in \rset^n$ is some initial starting point and
$\{\mu\}_{k \ge 0}$ is a nonincreasing positive sequence of
stepsizes.  We consider that the algorithm SPP returns either the
last point $x^k$ or  the average point $\hat{x}^k =
\frac{1}{\sum_{i=0}^{k-1} \mu_i}\sum_{i=0}^{k-1} \mu_i x^i$ when it
is called as a subroutine. Since the update rule of the positive
smoothing (stepsize) sequence $\{\mu_k\}_{k \ge 0}$ strongly
contributes to the convergence of the scheme, we discuss in the
following sections the most advantageous choices. We first prove the
following useful auxiliary result:

\begin{lemma}
\label{lemma_gradmiu_grad} Let $\mu \!> 0$, $S \!\in \Omega$. Then,
for any $g_{f}(x;S) \in \partial f(x;S)$, the following  holds:
\begin{equation*}
\norm{\nabla f_{\mu}(x;S)} \le \norm{g_{f}(x;S)} \qquad \forall x
\in \rset^n.
\end{equation*}
\end{lemma}

\begin{proof}
The optimality condition of  problem $\min\limits_{z \in \rset^n}
f(z;S) + \frac{1}{2\mu}\norm{x-z}^2$ is given by:
\begin{equation*}
  \frac{1}{\mu}\left( x - z_{\mu}(x;S)\right) \in \partial f(z_{\mu}(x;S);S).
\end{equation*}
The above inclusion easily implies that there is
$g_f(z_{\mu}(x;S);S) \in \partial f(z_{\mu}(x;S);S)$  such that:
\begin{align*}
\frac{1}{\mu}\norm{z_{\mu}(x;S)-x}^2
& =  \langle g_f(z_{\mu}(x;S);S), x - z_{\mu}(x;S)\rangle \\
& = \langle g_f(x;S), x - z_{\mu}(x;S)\rangle + \langle
g_f(z_{\mu}(x;S);S) - g_f(x;S),
x -z_{\mu}(x;S)\rangle \\
& \le \langle g_f(x;S), x - z_{\mu}(x;S)\rangle,
\end{align*}
where in the last inequality we used the convexity of $f$. Lastly,
by applying the Cauchy-Schzwarz inequality in the right hand side we
get the above statement.
\end{proof}

\noindent The following two well-known inequalities, which can be
found in \cite{Bul:03}, will be also useful in the sequel:
\begin{lemma}[Bernoulli] \label{lemma_bernoulli}
Let $t \in [0,1]$ and $x \in  [-1, \infty)$, then the following
holds:
\begin{equation*}
(1 + x)^t \le 1 + tx.
\end{equation*}
\end{lemma}

\begin{lemma}[Minkowski] \label{lemma_minkowski}
Let $x$  and $y$ be two random variables. Then, for any $1 \le p <
\infty$, the following inequality holds:
\begin{equation*}
\left( \mathbb{E}[|x + y|^p ] \right)^{1/p} \le \left(
\mathbb{E}[|x|^p] \right)^{1/p}+ \left(
\mathbb{E}[|y|^p]\right)^{1/p}.
\end{equation*}
\end{lemma}


\section{Nonasymptotic complexity analysis of SPP: convex objective function}
In this section analyze, under  Assumptions \ref{assump_fsmooth} and
\ref{assump_feasset}, the iteration complexity of SPP scheme with
nonincreasing stepsize rule to approximately solve the optimization
problem \eqref{convexfeas_random}. In order to prove this
nonasymptotic  result, we first define $\hat{\mu}_{1,k} =
\sum\limits_{i=0}^{k-1} \mu_i$ and $\hat{\mu}_{2,k} =
\sum\limits_{i=0}^{k-1} \mu_i^2$. Moreover, denote by ${\mathcal F}_{k}$
the history of random choices $\{S_k\}_{k\ge 0}$, i.e. ${\mathcal F}_{k}=\{S_0, \cdots, S_{k}\}$.

\begin{lemma}\label{lemma_convexcase}
Let Assumptions \ref{assump_fsmooth} and \ref{assump_feasset} hold
and  the sequences $\{x^k, y^k\}_{ k \ge 0}$ be generated by SPP
scheme with  positive stepsize $\{\mu_k\}_{k \ge 0}$. If we  define
the average sequences $\hat{x}^k  =
\frac{1}{\hat{\mu}_{1,k}}\sum_{i=0}^{k-1} \mu_i x^i$ and $\hat{y}^k
= \frac{1}{\hat{\mu}_{1,k}}\sum_{i=0}^{k-1} \mu_i y^i$, then  the
following relation holds:
\begin{align*}
\mathbb{E}\left[\text{dist}_{X_{S_k}}^2 (\hat{y}^k) \right] \ge
\frac{1}{\kappa} \mathbb{E}\left[\text{dist}_{X}^2
(\hat{x}^k)\right] - \frac{\hat{\mu}_{2,k}}{\hat{\mu}_{1,k}}
\sqrt{\mathbb{E}[\text{dist}_{X}^2(\hat{x}^k)]} \; \sqrt{\mathbb{E}[L_{f,S}^2]}.
 \end{align*}
\end{lemma}

\begin{proof}
By using the convexity of the function $\mathbb{I}_{\mu,S}(x) = \frac{1}{2\mu}\text{dist}_{X_S}^2(x)$ and taking the
conditional expectation w.r.t. $S_k$ over the history ${\mathcal F}_{k-1}=\{S_0, \cdots, S_{k-1}\}$, we get:
\begin{align*}
\mathbb{E}[\mathbb{I}_{1,S_k}(\hat{y}^k)|\mathcal{F}_{k-1}]  \ge \mathbb{E}\left[\mathbb{I}_{1,S_k}(\hat{x}^k)
+ \langle \nabla \mathbb{I}_{1,S_k}(\hat{x}^k), \hat{y}^k - \hat{x}^k\rangle|\mathcal{F}_{k-1}\right].
\end{align*}
Taking further the expectation over $\mathcal{F}_{k-1}$ we obtain:
\begin{align*}
\mathbb{E}[\mathbb{I}_{1,S_k}(\hat{y}^k)]
&  \ge \mathbb{E}\left[\mathbb{I}_{1,S_k}(\hat{x}^k)\right] + \mathbb{E}\left[\langle \nabla \mathbb{I}_{1,S_k}(\hat{x}^k), \hat{y}^k - \hat{x}^k\rangle\right] \nonumber \\
& \ge \mathbb{E}\left[\mathbb{I}_{1,S_k}(\hat{x}^k)\right] + \frac{\mathbb{E}\left[\langle \nabla \mathbb{I}_{1,S_k}(\hat{x}^k), \sum\limits_{i=0}^{k-1} \mu_i^2 \nabla f_{\mu_i}(x^i;S_i)\rangle \right ]}{\hat{\mu}_{1,k}} \nonumber \\
& \ge \mathbb{E}\left[ \mathbb{I}_{1,S_k}(\hat{x}^k)\right] - \mathbb{E}\left[\frac{\hat{\mu}_{2,k}}{\hat{\mu}_{1,k}} \norm{\nabla \mathbb{I}_{1,S_k}(\hat{x}^k)} \left\|\sum\limits_{i=0}^{k-1}
\frac{\mu_i^2}{\hat{\mu}_{2,k}} \nabla f_{\mu_i}(x^i;S_i)\right\|\right] \nonumber\\
& \ge \mathbb{E}\left[ \mathbb{I}_{1,S_k}(\hat{x}^k)\right] - \frac{\hat{\mu}_{2,k}}{\hat{\mu}_{1,k}} \mathbb{E}\left[\norm{\nabla \mathbb{I}_{1,S_k}(\hat{x}^k)} \sum\limits_{i=0}^{k-1} \frac{\mu_i^2}{\hat{\mu}_{2,k}}
\left\| \nabla f_{\mu_i}(x^i;S_i)\right\|\right].
\end{align*}
Further, using Lemma  \ref{lemma_gradmiu_grad}, Assumption \ref{assump_feasset} and Cauchy-Schwartz inequality, we have:
\begin{align*}
\mathbb{E}\left[
\frac{1}{2}\text{dist}_{X_{S_k}}^2(\hat{y}^k)\right] &
\overset{\text{Lemma} \; \ref{lemma_gradmiu_grad}}{\ge}
\mathbb{E}\left[
\frac{1}{2}\text{dist}_{X_{S_k}}^2(\hat{x}^k)\right] -
\frac{\hat{\mu}_{2,k}}{2\hat{\mu}_{1,k}} \mathbb{E}\left[\text{dist}_{X_{S_k}}(\hat{x}^k) L_{f,S_k}\right] \nonumber\\
& \overset{\text{Assump.} \; \ref{assump_feasset}}{\ge}
\frac{1}{2\kappa}\mathbb{E}\left[\text{dist}_{X}^2(\hat{x}^k)\right]
- \frac{\hat{\mu}_{2,k}}{2\hat{\mu}_{1,k}}
\sqrt{\mathbb{E}[\text{dist}_{X}^2(\hat{x}^k)]}\sqrt{\mathbb{E}[ L_{f,S}^2]},
\end{align*}
which proves the statement of the lemma.
\end{proof}

\noindent Now, we are ready to derive the convergence rate of SPP in
the average sequence  $\hat{x}^k$:

\begin{theorem}\label{th_convexcase}
Under Assumptions \ref{assump_fsmooth} and \ref{assump_feasset}, let
the sequence $\{x^k\}_{k \ge 0}$ be generated by the algorithm SPP
with nonincreasing positive stepsize $\{\mu_k\}_{k \ge 0}$. Define
the average sequence $\hat{x}^k = \frac{1}{\hat{\mu}_{1,k}}\sum\limits_{i=0}^{k-1}\mu_i x^i$ and
$\mathcal{R}_{\mu} =  \mu_0\kappa(\|x^0 - x^*\|^2 + \mathbb{E}[L_{f,S}^2]\hat{\mu}_{2,k})$. Then,
the following estimates for suboptimality and feasibility violation hold:
\begin{equation}
\label{thconvexcase_subopt}
\begin{split}
 - \kappa \mathbb{E}[L_{f,S}^2] & \left(\frac{\hat{\mu}_{2,k}}{\hat{\mu}_{1,k}}
\!+\!  2\mu_0 \!\right)   -  \sqrt{\mathbb{E}[L_{f,S}^2] \
\frac{\mathcal{R}_{\mu}}{\hat{\mu}_{1,k}}}
\le   \mathbb{E}[F(\hat{x}^k)] - F^* \le  \frac{\mathcal{R}_{\mu} }{2\mu_0 \kappa \hat{\mu}_{1,k} }\\
&\quad  \mathbb{E}[\text{dist}_{X}^2(\hat{x}^k)]
 \le 2\kappa^2 \mathbb{E}[L_{f,S}^2]  \left(\frac{\hat{\mu}_{2,k}}{\hat{\mu}_{1,k}}+2\mu_0\right)^2
+  \frac{2\mathcal{R}_{\mu}}{\hat{\mu}_{1,k}}.
\end{split}
\end{equation}
\end{theorem}

\begin{proof} Since the function  $z \to f(z;S)+ \frac{1}{2\mu}\|z - x\|^2$ is strongly convex, we have:
\begin{align}
 f(z;S) + \frac{1}{2\mu}\norm{z-x}^2
& \ge  f(z_{\mu}(x;S);S)  + \frac{1}{2\mu}\norm{z_{\mu}(x;S)-x}^2 + \frac{1}{2\mu}\norm{z_{\mu}(x;S) - z}^2 \nonumber \\
& = f_{\mu}(x;S) + \frac{1}{2\mu}\norm{z_{\mu}(x;S) - z}^2   \quad \forall z \in \rset^n. \label{th_convexcase_scf}
\end{align}
By taking $x = x^k, S=S_k, z = x^*, \mu = \mu_k$ in
\eqref{th_convexcase_scf} and using  the strictly nonexpansive property
of the projection operator, see e.g. \cite{Ned:11}:
\begin{equation}\label{strict_nonexpansiveness}
 \norm{x - [x]_{X_{S_k}}}^2 \le \norm{x - z}^2 - \norm{z - [x]_{X_{S_k}}}^2 \qquad \forall z \in X_{S_k}, \; x \in \rset^n,
\end{equation}
then these lead to:
\begin{align}\label{th_convexcase_scf2}
f(x^*;S_k) + & \frac{1}{2\mu_k}\norm{x^k-x^*}^2
\ge  f_{\mu_k}(x^k;S_k) + \frac{1}{2\mu_k} \norm{y^k - x^*}^2 \nonumber \\
& \ge  f_{\mu_k}(x^k;S_k) + \frac{1}{2\mu_k} \norm{[y^k]_{X_{S_k}} -x^*}^2  + \frac{1}{2\mu_k}\norm{y^k - [y^k]_{X_{S_k}}}^2 \nonumber \\
& =  f_{\mu_k}(x^k;S_k) + \frac{1}{2\mu_k} \norm{x^{k+1} - x^*}^2 + \frac{1}{2\mu_k}\norm{y^k - x^{k+1}}^2.
\end{align}
By denoting $\mathbb{I}_{\mu,S}(x) = \frac{1}{2\mu}\norm{x-[x]_{X_S}}^2$, from \eqref{th_convexcase_scf2}, it can be easily seen that:
\begin{align}
& \mu_k(f(x^k;S_k)- f(x^*;S_k))+ \mathbb{I}_{1,S_k}(y^k) - \frac{\mu_k^2}{2}L_{f,S_k}^2 \nonumber\\
&\le \mu_k(f(x^k;S_k)- f(x^*;S_k)) + \mathbb{I}_{1,S_k}(y^k) - \frac{\mu_k^2}{2}\norm{\nabla f(x^k;S_k)}^2 \nonumber\\
& = \mu_k(f(x^k;S_k)- f(x^*;S_k))+ \mathbb{I}_{1,S_k}(y^k)  + \min_{z \in \rset^n} \left[ \mu_k \langle \nabla f(x^k;S_k), z - x^k\rangle + \frac{1}{2}\norm{z-x^k}^2\right] \nonumber\\
&\le \mu_k(f(x^k;S_k)- f(x^*;S_k))+ \mathbb{I}_{1,S_k}(y^k)  + \mu_k \langle \nabla f(x^k;S_k), y^k-x^k\rangle + \frac{1}{2}\norm{y^k-x^k}^2 \nonumber\\
& = \mu_k(f(x^k;S_k) + \langle \nabla f(x^k;S_k), y^k-x^k\rangle + \frac{1}{2\mu_k}\norm{y^k-x^k}^2 - f(x^*;S_k)) + \mathbb{I}_{1,S_k}(y^k) \nonumber\\
& \overset{\text{conv. f}}{\le} \mu_k  (f_{\mu_k}(x^k;S_k)  - f(x^*;S_k)) + \mathbb{I}_{1,S_k}(y^k)  \nonumber \\
& \overset{\eqref{th_convexcase_scf2}}{\le} \frac{1}{2}\norm{x^{k} - x^*}^2 - \frac{1}{2}\norm{x^{k+1} - x^*}^2. \nonumber
\end{align}
Taking now the conditional expectation in $S_k$ w.r.t. the history ${\mathcal F}_{k-1}=\{S_0, \cdots, S_{k-1}\}$  in the last inequality we have:
\begin{align*}
\mu_k(F(x^k) - F(x^*)) + \mathbb{E}[\mathbb{I}_{1,S_k}(y^k)|{\mathcal F}_{k-1}] & - \frac{\mu_k^2}{2}\mathbb{E}[L_{f,S_k}^2]
 \le \frac{1}{2}\norm{x^{k} - x^*}^2 - \frac{1}{2}\mathbb{E}[\norm{x^{k+1} - x^*}^2|{\mathcal F}_{k-1}].
\end{align*}
Taking further the  expectation over ${\mathcal F}_{k-1}$ and summing  over $i=0, \cdots, k-1$, results in:
\begin{align}
\frac{\norm{x^{0} - x^*}^2}{2\sum\limits_{i=0}^{k-1} \mu_i}
&\ge \frac{1}{\sum\limits_{i=0}^{k-1} \mu_i}\sum\limits_{i=0}^{k-1} \mathbb{E}[\mu_i(F(x^i) - F(x^*))] + \mathbb{E}[\mathbb{I}_{1,S}(y^i)] - \frac{\mu_i^2}{2}\mathbb{E}[L_{f,S}^2] \nonumber\\
& =  \frac{1}{\sum\limits_{i=0}^{k-1} \mu_i}\sum\limits_{i=0}^{k-1} \mathbb{E}[\mu_i(F(x^i)- F(x^*))] + \mu_i\mathbb{E}[\mathbb{I}_{\mu_i,S}(y^i)] - \frac{\mu_i^2}{2}\mathbb{E}[L_{f,S}^2] \nonumber\\
& \ge  \frac{1}{\sum\limits_{i=0}^{k-1} \mu_i}\sum\limits_{i=0}^{k-1} \mathbb{E}[\mu_i(F(x^i)- F(x^*))] + \mu_i\mathbb{E}[\mathbb{I}_{\mu_0,S}(y^i)] - \frac{\mu_i^2}{2}\mathbb{E}[L_{f,S}^2] \nonumber\\
& \overset{\text{Jensen}}{\ge} \mathbb{E}[F(\hat{x}^k)- F(x^*)] + \mathbb{E}[\mathbb{I}_{\mu_0,S}(\hat{y}^k)] - \frac{\mathbb{E}[L_{f,S}^2] \hat{\mu}_{2,k} }{2\hat{\mu}_{1,k}}
\label{th_convexcase_prelimbound}
\end{align}
The relation \eqref{th_convexcase_prelimbound} implies the following upper bound on the suboptimality gap:
\begin{align}
\label{th_convexcase_suboptup}
\mathbb{E}[F(\hat{x}^k)- F(x^*)] \le \frac{\norm{x^{0} - x^*}^2 + \mathbb{E}[L_{f,S}^2]\hat{\mu}_{2,k} }{2\hat{\mu}_{1,k}}.
\end{align}
On the other hand, recalling $\nabla F(x^*) = \mathbb{E}[\nabla f(x^*;S)]$, we use the following fact:
\begin{align}
\mathbb{E}[F(\hat{x}^k)] - F(x^*)
&\ge  \mathbb{E}[\langle \nabla F(x^*), \hat{x}^k - x^* \rangle] \nonumber\\
&\ge  \mathbb{E}[\langle \nabla F(x^*), [\hat{x}^k]_{X} - x^* \rangle] +  \mathbb{E}[\langle \nabla F(x^*), \hat{x}^k - [\hat{x}^k]_{X} \rangle] \nonumber\\
& \ge  - \mathbb{E}[L_{f,S}] \mathbb{E}[\text{dist}_X(\hat{x}^k)] \nonumber \\
&\overset{\text{Jensen}}{\ge}  -  \sqrt{\mathbb{E}[L_{f,S}^2] \ \mathbb{E}[\text{dist}_X^2(\hat{x}^k)]} \qquad \forall k \ge 0, \label{th_convexcase_rel2}
\end{align}
which is derived from the optimality conditions $\langle \nabla F(x^*), z-x^*\rangle \ge 0$ for all $z \in X$, the Cauchy-Schwarz and Jensen inequalities.
By denoting $r_0 = \norm{x^0-x^*}$ and combining \eqref{th_convexcase_prelimbound} with Lemma \ref{lemma_convexcase} and the last inequality \eqref{th_convexcase_rel2}, we obtain:
\begin{align*}
\mathbb{E}[\text{dist}_{X}^2(\hat{x}^k)]   - \kappa \sqrt{\mathbb{E}[L_{f,S}^2]}\left(\frac{\hat{\mu}_{2,k}}{\hat{\mu}_{1,k}}+2\mu_0\right) & \sqrt{\mathbb{E}[\text{dist}_X^2(\hat{x}^k)]}  \le \frac{\mu_0 \kappa r_0^2 + \mu_0 \kappa \mathbb{E}[L_{f,S}^2]\hat{\mu}_{2,k}}{\hat{\mu}_{1,k}}.
\end{align*}
This relation clearly implies an upper bound on the feasibility residual:
\begin{align}
\label{th_convexcase_feas_bound}
\sqrt{\mathbb{E}[\text{dist}_{X}^2(\hat{x}^k)]}
& \le\! \kappa \sqrt{\mathbb{E}[L_{f,S}^2]}\left(\frac{\hat{\mu}_{2,k}}{\hat{\mu}_{1,k}}+2\mu_0\right) \! + \! \sqrt{\frac{\mu_0 \kappa r_0^2 + \mu_0 \kappa \mathbb{E}[L_{f,S}^2]\hat{\mu}_{2,k}} { \hat{\mu}_{1,k}}}.
\end{align}
\noindent Also, combining \eqref{th_convexcase_rel2} and \eqref{th_convexcase_feas_bound} we obtain the lower bound on the suboptimality gap:
\begin{equation}\label{th_convexcase_suboptlow}
\mathbb{E}[F(\hat{x}^k)] - F^* \!\ge\! - \kappa \mathbb{E}[L_{f,S}^2] \left(\frac{\hat{\mu}_{2,k}}{\hat{\mu}_{1,k}}\!+\!2\mu_0\right) \! - \!
\sqrt{\mathbb{E}[L_{f,S}^2]} \sqrt{\frac{\mu_0 \kappa r_0^2 \!+\! \mu_0 \kappa \mathbb{E}[L_{f,S}^2]\hat{\mu}_{2,k}} { \hat{\mu}_{1,k}}}.
\end{equation}
From the upper and lower  suboptimality bounds \eqref{th_convexcase_suboptup}-\eqref{th_convexcase_suboptlow} and feasibility bound \eqref{th_convexcase_feas_bound}, we deduce our convergence rate results.
\end{proof}

\noindent Note that the suboptimality bound \eqref{th_convexcase_suboptup} obtained for the SPP algorithm coincides with the one given for the standard subgradient method \cite{Nes:04}.
Further, we provide the convergence estimates for the algorithm SPP with constant stepsize for a desired accuracy $\epsilon$. For simplicity, assume that $r_0=\norm{x^0-x^*}\ge 1$ and $\mathbb{E}[L_{f,S}^2] \ge 2$.

\begin{corrolary}\label{corr_convexcase}
Under the assumptions of Theorem \ref{th_convexcase}, let $\{x^k\}_{k \ge 0}$ be the sequence generated by algorithm SPP with constant stepsize  $\mu>0$.
Also let $\epsilon > 0$ be the desired accuracy, $K$ be an integer satisfying:
$$ K \ge \frac{\mathbb{E}[L_{f,S}^2] \norm{x^0-x^*}^2}{\epsilon^2} \max\left\{1, (3\kappa + \sqrt{2\kappa})^2\right\},$$
and the stepsize  be chosen as:
\begin{equation*}
 \mu = \frac{\epsilon}{\mathbb{E}[L_{f,S}^2] (3\kappa + \sqrt{2\kappa})}.
\end{equation*}
Then, after $K$ iterations, the average point $\hat{x}^K = \frac{1}{K} \sum\limits_{i=0}^{K-1} x^i$ satisfies:
$$ \left| \mathbb{E}[F(\hat{x}^K)] - F^* \right| \le \epsilon \quad \text{and} \quad \sqrt{\mathbb{E}[\text{dist}_X^2(\hat{x}^K)]} \le \epsilon.$$
\end{corrolary}

\begin{proof}
We consider $k = K$ in Theorem \ref{th_convexcase} and, by taking into account that $\mu_k = \mu$ for all $k \ge 0$,
we aim to obtain the lowest value of the right hand side of \eqref{thconvexcase_subopt}
by minimizing over $\mu > 0$. Thus, by recalling that $r_0 = \norm{x^0-x^*}$, we obtain for the optimal smoothing parameter:
$$ \mu = \sqrt{\frac{r_0^2}{K \mathbb{E}[L_{f,S}^2]}}$$
the optimal rate
\begin{equation}\label{subopt_upperbound}
 \mathbb{E}[F(\hat{x}^K)] - F^*  \le \sqrt{\frac{\mathbb{E}[L_{f,S}^2] r_0^2}{K}}.
\end{equation}
Also using the optimal parameter $\tilde{\mu}$ into the other relations of Theorem \ref{th_convexcase} results:
\begin{equation}\label{feas_bound}
\mathbb{E}[\text{dist}_{X}^2(\hat{x}^K)]
\le \frac{r_0^2}{K}\left(18 \kappa^2 + 4\kappa \right)
\end{equation}
and
\begin{equation}\label{subopt_lowerbound}
 \mathbb{E}[F(\hat{x}^K)] -F^* \ge - (3\kappa + \sqrt{2\kappa})\sqrt{\frac{\mathbb{E}[L_{f,S}^2]r_0^2 }{K}}.
\end{equation}
From the upper and lower  suboptimality bounds \eqref{subopt_upperbound}-\eqref{subopt_lowerbound} and feasibility bound \eqref{feas_bound}, we
deduce the following bound:
$$ K \ge \frac{ \mathbb{E}[L_{f,S}^2]r_0^2}{\epsilon^2} \max\left\{1, (3\kappa + \sqrt{2\kappa})^2\right\}$$
which confirms our result.
\end{proof}

\noindent In conclusion,   Corollary \ref{corr_convexcase}
states that for a desired accuracy $\epsilon$, if we choose a
constant stepsize $\mu = \mathcal{O} (\epsilon)$ and perform a
number of SPP iterations  $\mathcal{O}\left(
\frac{1}{\epsilon^2}\right)$ we obtain an $\epsilon$-optimal
solution for our original stochastic constrained convex problem
\eqref{convexfeas_random}. Note that for convex problems with
objective function having bounded subgradients the previous
convergence estimates derived for the SPP algorithm are similar to
those corresponding to the classical deterministic proximal point method
\cite{Gul:91} and  subgradient method  \cite{Nes:04}.


\section{Nonasymptotic complexity analysis of SPP: strongly convex objective function}
In this section we analyze the convergence behavior  of the  SPP
scheme under smoothness and strong convexity assumptions on the
objective function of constrained  problem
\eqref{convexfeas_random}. Therefore, in this  section the
Assumption \ref{assump_fsmooth} is replaced by the following
assumptions:
\begin{assumption}\label{assump_strongconvex}
Each function $f(\cdot;S)$ is differentiable and $\sigma_{f,S}$-restricted strongly convex, that is there exists strong convexity constant $\sigma_{f,S} \ge 0$ such that:
\begin{equation*}
     f(x;S) \ge f(y;S)  + \langle \nabla f(y;S), x-y\rangle + \frac{\sigma_{f,S}}{2}\norm{x-y}^2 \qquad \forall x,y \in \rset^n.
\end{equation*}
Moreover, the strong convexity constants $\sigma_{f,S}$ satisfy $\sigma_F = \mathbb{E}[\sigma_{f,S}]  > 0$.
\end{assumption}

\noindent Notice that if for some function $f(\cdot;S)$ the
corresponding constant $\sigma_{f,S}$ is equal to $0$, then
$f(\cdot;S)$ is only convex function. However, relation
$\mathbb{E}[\sigma_{f,S}] = \sigma_F >0 $ implies that the objective
function $F$ of problem \eqref{convexfeas_random} is strongly convex
with constant $\sigma_F > 0$. In the sequel we will analyze the SPP
scheme  also under the following smoothness assumption:

\begin{assumption}\label{assump_lipgrad}
Each function $f(\cdot;S)$ has Lipschitz gradient, that is  there exists Lipschitz constant $L_{f,S} > 0$ such that:
\begin{equation*}
\norm{\nabla f(x;S) - \nabla f(y;S)} \le L_{f,S} \norm{x-y} \qquad \forall x,y \in \rset^n.
\end{equation*}
\end{assumption}

\noindent Note that Assumptions \ref{assump_strongconvex} and
\ref{assump_lipgrad} are standard for the convergence analysis of
SPP like schemes, see e.g. \cite{MouBac:11,RyuBoy:16}. We first
present an auxiliary result on the behavior of the proximal mapping
$z_{\mu}(\cdot;S)$.

\begin{lemma} \label{lemma_non-as}
Let $f(\cdot;S)$ satisfy Assumption \ref{assump_strongconvex}. Further, for any $S \in \Omega$ and $\mu > 0$, we define $\theta_S(\mu) =  \frac{1}{1 + \mu\sigma_{f,S}}$. Then, the following contraction inequality holds for the prox operator:
\begin{equation*}
\left\|z_{\mu}(x;S) - z_{\mu}(y;S) \right\| \le \theta_S(\mu) \norm{x-y} \qquad \forall x,y \in \rset^n.
\end{equation*}
\end{lemma}

\begin{proof}
Let $\sigma_{f,S} \geq 0$ be the strong convexity constant of the
function $f(\cdot;S)$. Notice that we allow the convex case, that is
$\sigma_{f,S} = 0$ for some $S$. Then, it is known that the Moreau
approximation $f_{\mu}(\cdot;S)$ is also a  $\hat{\sigma}_{f,S}$-strongly
convex function  with strong convexity constant,  see e.g.
\cite{RocWet:98}:
\[ \hat{\sigma}_{f,S} = \frac{\sigma_{f,S}}{1 + \mu\sigma_{f,S}}. \]
Clearly, in the simple convex case, that is  $\sigma_{f,S} = 0$, we
also have $\hat{\sigma}_{f,S} = 0$. By denoting $\hat{L}_{f,S} =
\frac{1}{\mu}$ the Lipschitz constant of the gradient of
$f_{\mu}(\cdot;S)$, the following well-known relation holds  for the smooth
and (strongly) convex function $f_{\mu}(\cdot;S)$, see e.g. \cite{Nes:04}:
\begin{align}
\langle \nabla f_{\mu}(x;S) - \nabla f_{\mu}(y;S),x - y \rangle  & \ge \frac{1}{\hat{\sigma}_{f,S} + \hat{L}_{f,S}}\norm{\nabla f_{\mu}(x;S)
- \nabla f_{\mu}(y;S)}^2 \nonumber \\
& \qquad +  \frac{\hat{\sigma}_{f,S}\hat{L}_{f,S}}{\hat{L}_{f,S}  + \hat{\sigma}_{f,S}}\norm{x - y}^2 \quad \forall x,y \in \rset^n.
\label{coercivity}
\end{align}
By using Assumption \ref{assump_strongconvex}, then it can be also obtained that:
\begin{equation}\label{conseq1_coerciv}
   \norm{\nabla f_{\mu}(x;S) - \nabla f_{\mu}(y;S)} \ge \hat{\sigma}_{f,S} \norm{x-y}\qquad \forall x,y \in \rset^n.
\end{equation}
Using this relation, we further derive that:
\begin{align*}
& \norm{z_{\mu}(x;S)-z_{\mu}(y;S)}^2
 = \norm{x - y + \mu(\nabla f_{\mu}(y;S) - \nabla f_{\mu}(x;S) )}^2 \\
&  = \norm{x-y}^2 + 2\mu\langle \nabla f_{\mu}(y;S) - \nabla f_{\mu}(x;S), x-y\rangle + \mu^2\norm{\nabla f_{\mu}(x;S)-\nabla f_{\mu}(y;S)}^2 \\
& \overset{\eqref{coercivity}}{\le} \left( 1 - \frac{2\mu\hat{\sigma}_{f,S} \hat{L}_{f,S}}{\hat{L}_{f,S} + \hat{\sigma}_{f,S}}\right)\norm{x-y}^2 +
\mu\left(\mu - \frac{2}{\hat{L}_{f,S} + \hat{\sigma}_{f,S}}\right)\norm{\nabla f_{\mu}(x;S) - \nabla f_{\mu}(y;S)}^2 \\
& \overset{\eqref{conseq1_coerciv}}{\le} \left[ 1 + \hat{\sigma}_{f,S}^2\left(\mu^2 - \frac{2\mu}{\hat{\sigma}_{f,S} + \hat{L}_{f,S}} \right)
- \frac{2\mu\hat{\sigma}_{f,S} \hat{L}_{f,S}}{\hat{L}_{f,S} + \hat{\sigma}_{f,S}}   \right]\norm{x-y}^2 \\
& = \left(1 - \hat{\sigma}_{f,S}\mu\right)^2 \norm{x-y}^2 \qquad \forall x,y \in \rset^n,
\end{align*}
which implies our result.
\end{proof}

\noindent  Notice that if all the functions $f(\cdot;S)$ are
just convex, that is  they satisfy Assumption
\ref{assump_strongconvex} with $\sigma_{f,S} = 0$, then Lemma
\ref{lemma_non-as} highlights the nonexpansiveness property of the
proximal operator $z_{\mu}(\cdot;S)$. We will further keep using the
notation $\theta_S(\mu)$ for the contraction factor of the operator
$z_{\mu}(\cdot;S)$. Moreover, in all our proofs below, regarding the results in expectation, we use the standard
technique of tacking first expectation with respect to $S_k$ conditioned on  ${\mathcal F}_{k-1}$ and then take the  expectation over the entire history ${\mathcal F}_{k-1}$ (see the proof of Theorem \ref{th_convexcase}). For simplicity of the exposition and for saving space, we  omit these details below.


\subsection{Linear convergence to noise dominated region for constant stepsize SPP}
\noindent Next we analyze the sequence  generated by the SPP scheme
with constant stepsize $\mu>0$ and provide a nonasymptotic bound on
the quadratic mean $\mathbb{E}[\norm{x^k-x^*}^2]$.

\begin{theorem}\label{th_conststep}
Under Assumption \ref{assump_strongconvex}, let the sequence
$\{x^k\}_{k \ge 0}$ be generated by the algorithm SPP with constant
stepsize $\mu > 0$. Further,  assume $\sigma_{f}^{\max} =
\sup\limits_{S \in \Omega} \sigma_ {f,S} < \infty$. Then,
$\mathbb{E}[\theta_S^2(\mu)] \le \mathbb{E}[\theta_S(\mu)] < 1$ and
the following linear convergence to some region around the optimal
point in  the quadratic mean holds:
\begin{align*}
\mathbb{E}[\norm{x^k  -x^*}^2] \le
2\left(\mathbb{E}\left[\theta_S^2(\mu) \right]\right)^{k}
\norm{x^0-x^*}^2 +  \frac{2\mu^2 \mathbb{E}[\norm{\nabla
f(x^*;S)}^2]}{\left(1 - \sqrt{\mathbb{E}[\theta_S^2(\mu)]}
\right)^2}.
\end{align*}
\end{theorem}

\begin{proof}
First, it can be easily seen that for any $\mu>0$ and $S \in \Omega$
we have $\theta_S^2(\mu) \leq \theta_S(\mu) \leq 1$ and assuming
that $\sigma_{f}^{\max}  < \infty$ we obtain:
\begin{equation*}
0 \le \mathbb{E}[\theta_S^2(\mu)] \le \mathbb{E}[\theta_S(\mu)] =
\mathbb{E}\left[\frac{1}{1 + \mu\sigma_{f,S}}\right] \le 1 -
\mathbb{E}\left[\frac{\mu \sigma_{f,S}}{1 + \mu \sigma_{f,S}}\right]
\le 1 - \frac{\mu \sigma_F}{1 + \mu \sigma_{f}^{\max}} < 1.
\end{equation*}
Then, by applying Lemma \ref{lemma_non-as} with $S = S_k, x = x^k$ and $z = x^*$, results in:
\begin{equation*}
\left\|z_{\mu}(x^{k};S_k) - z_{\mu}(x^*;S_k) \right\| \le \theta_{S_k}(\mu)\norm{x^k - x^*},
\end{equation*}
which, by the triangle inequality, further implies:
\begin{equation*}
\left\|z_{\mu}(x^{k};S_k) - x^* \right\| \le  \theta_{S_k}(\mu) \norm{x^k -x^*}  + \norm{z_{\mu}(x^*;S_k) -x^*}.
\end{equation*}
By using the nonexpansiveness property of the projection operator we get that $\norm{x^{k+1}-x^*} \le \norm{y^k-x^*}$, then the last inequality leads to the reccurent relation:
\begin{equation}\label{root_descent}
\left\|x^{k+1} - x^* \right\| \le \left\|z_{\mu}(x^{k};S_k) - x^* \right\|  \le \theta_{S_k}(\mu) \norm{x^k -x^*}  + \norm{z_{\mu}(x^*;S_k) -x^*}.
\end{equation}
The relation \eqref{root_descent}, Minkowski inequality and Lemma
\ref{lemma_gradmiu_grad} lead to the following recurrence:
\begin{align*}
 \sqrt{\mathbb{E}[\left\|x^{k+1} - x^* \right\|^2]} & \overset{\eqref{root_descent}}{\le} \sqrt{\mathbb{E}\left[\left(\theta_{S_k}(\mu)\norm{x^k -x^*}  + \norm{z_{\mu}(x^*;S_k) -x^*}\right)^2\right]} \\
& \overset{\text{Lemma} \; \ref{lemma_minkowski}}{\le} \sqrt{\mathbb{E}\left[\theta_{S_k}^2(\mu)
\norm{x^k -x^*}^2 \right]}  + \sqrt{\mathbb{E} \left[\norm{z_{\mu}(x^*;S_k) -x^*}^2\right]} \\
& = \sqrt{\mathbb{E}\left[\theta_S^2(\mu)\right]} \sqrt{\mathbb{E}\left[\norm{x^k -x^*}^2\right]}   + \mu\sqrt{\mathbb{E}\left[\norm{\nabla f_{\mu}(x^*;S)}^2\right]} \\
& \overset{\text{Lemma} \; \ref{lemma_gradmiu_grad}}{\le} \sqrt{\mathbb{E}\left[\theta_S^2(\mu)\right]} \sqrt{\mathbb{E}[\norm{x^k -x^*}^2]}  + \mu\sqrt{\mathbb{E}\left[\norm{\nabla f(x^*;S)}^2\right]}.
\end{align*}
This  yields the following  relation valid for all $\mu > 0$ and $k \ge 0$:
\begin{equation}\label{th_conststep_reccurence}
\sqrt{\mathbb{E}[\left\|x^{k+1} - x^* \right\|^2]} \le \sqrt{\mathbb{E}\left[\theta_S^2(\mu)\right]} \sqrt{\mathbb{E}[\norm{x^k -x^*}^2]}
 + \mu\sqrt{\mathbb{E}\left[\norm{\nabla f(x^*;S)}^2\right]},
\end{equation}

\noindent Denote $r_k \!=\! \sqrt{\mathbb{E}[\norm{x^k -x^*}^2]}, \;
\eta \!=\! \sqrt{\mathbb{E}\left[\norm{\nabla f(x^*;S)}^2\right]]}$
and $\theta(\mu) \!=\!
\sqrt{\mathbb{E}\left[\theta_S^2(\mu)\right]}$. Then, we get:
\begin{equation*}
r_{k+1} \le \theta(\mu) r_k + \mu\eta.
\end{equation*}
Finally, a simple inductive argument leads to:
\begin{align*}
r_{k}
&\le r_0 \theta(\mu)^{k}  + \mu \eta \left[1 + \theta(\mu) + \cdots + \theta(\mu)^{k-1} \right] \\
& =  r_0 \theta(\mu)^{k}  + \mu\eta \frac{1 - \theta(\mu)^{k}}{1 - \theta(\mu)}\\
& \le  r_0 \theta(\mu)^{k}  + \frac{\mu\eta}{1 - \theta(\mu)}.
\end{align*}
By squaring and returning to our basic notations, we recover our statement.
\end{proof}

\noindent Theorem \ref{th_conststep} proves a linear convergence
rate in expectation for the sequence generated by the  SPP algorithm
with a constant stepsize $\mu
> 0$   when the sequence $\{x^k\}_{k \ge 0}$ is outside of a
\textit{noise dominated} neighborhood of the optimal set  of radius
$\frac{\mu \sqrt{\mathbb{E}[\norm{\nabla f(x^*;S)}^2]}}{1 -
\sqrt{\mathbb{E}[\theta_S^2(\mu)]}}$. It also  establishes the
boundedness of the sequence $\{x^k\}_{k \ge 0}$ when the stepsize is
constant. Notice that in \cite{RyuBoy:16} a similar result has been
given  for an unconstrained optimization model with the difference
that the convergence rate was provided for
$\mathbb{E}[\norm{x^k-x^*}]$. However, our proof is simpler and more
elegant, based on the  properties of Moreau approximation, despite
the fact that we consider the constrained case.


\subsection{Nonasymptotic sublinear convergence rate of variable stepsize SPP}
\noindent In this section we derive sublinear
 convergence rate of order $\mathcal{O}(1/k^\gamma)$ for the variable stepsize
SPP scheme, in a nonasymptotic fashion. We first prove the
boundedness of $\{x^k\}_{k \ge 0}$ when the stepsize is
nonincreasing, which will be useful for the subsequent convergence
results.

\begin{lemma}\label{lemma_varstepbounded}
Under Assumption \ref{assump_strongconvex}, let the sequence
$\{x^k\}_{k \ge 0}$ be generated by the algorithm SPP with
nonincreasing positive stepsize $\{\mu_k\}_{k \ge 0}$. Then, the
following relation holds:
\begin{equation*}
\mathbb{E}\left[\|x^{k} - x^* \|\right] \le \sqrt{\mathbb{E}[\left\|x^{k} - x^* \right\|^2]} \le
\max\left\{\norm{x^0-x^*},\frac{\mu_0\sqrt{\mathbb{E}\left[\norm{\nabla f (x^*;S)}^2\right]}}{1 - \sqrt{\mathbb{E}\left[\theta_S^2(\mu_0)\right]} }\right\}.
\end{equation*}
\end{lemma}

\begin{proof}
By taking $\mu=\mu_k$ in relation \eqref{th_conststep_reccurence}, we obtain:
\begin{align*}
\sqrt{\mathbb{E}[\left\|x^{k+1} - x^* \right\|^2]} \le
\sqrt{\mathbb{E}\left[\theta_S^2(\mu_k)\right]}
\sqrt{\mathbb{E}[\norm{x^k -x^*}^2]}  +
\mu_k\sqrt{\mathbb{E}\left[\norm{\nabla f(x^*;S)}^2\right]}.
\end{align*}

\noindent By using the notations $r_k =
\sqrt{\mathbb{E}[\left\|x^{k} - x^* \right\|^2]}, \theta_k =
\sqrt{\mathbb{E}\left[\theta_S^2(\mu_k) \right]}$ and $\eta =
\sqrt{\mathbb{E}\left[\norm{\nabla f (x^*;S)}^2\right]}$, the last
inequality leads to:
\begin{align}
r_{k+1} & \le \theta_k r_k + \left(1 - \theta_k \right)
\frac{\mu_k}{1 - \theta_k} \eta \nonumber\\ & \le
\max\left\{r_k, \frac{\mu_k}{1 - \theta_k} \eta \right\} \le
\max\left\{r_0, \frac{\mu_0}{1 - \theta_0}
\eta,\cdots,\frac{\mu_k}{1 - \theta_k} \eta \right\}.
\label{seqbound_ineq}
\end{align}
By observing the  fact that $t \mapsto \mathbb{E}\left[\frac{\sigma_{f,S}}{(1+t\sigma_{f,S})^2}
+ \frac{\sigma_{f,S}}{1+ t\sigma_{f,S}} \right]$ is nonincreasing in $t$, and implicitly:
\begin{align*}
\frac{\mu_{k-1}}{1 - \theta_{k-1}}
&= \frac{1}{\mathbb{E}\left[\frac{\sigma_{f,S}}{(1+\mu_{k-1}\sigma_{f,S})^2} + \frac{\sigma_{f,S}}{1+ \mu_{k-1}\sigma_{f,S}} \right]} \\
&\ge \frac{1}{\mathbb{E}\left[\frac{\sigma_{f,S}}{(1+\mu_{k}\sigma_{f,S})^2} + \frac{\sigma_{f,S}}{1+ \mu_{k}\sigma_{f,S}} \right]}
= \frac{\mu_{k}}{1 - \theta_k},
\end{align*}
then we have $\max\limits_{0 \le i \le k} \frac{\mu_{i}}{1 -
\theta_i} = \frac{\mu_{0}}{1 - \theta_0}$ and the relation
\eqref{seqbound_ineq} becomes:
\begin{align}
r_{k}  \le \max\left\{r_0, \frac{\mu_0}{1 - \theta_0} \eta
\right\} \quad \forall k\ge 0,
\end{align}
which implies our result.
\end{proof}

\noindent Furthermore, we need an upper bound on the sequence
$\{\mathbb{E}[\norm{\nabla f(x^k;S)}]\}_{k \ge 0}$:

\begin{lemma}\label{lemma_gradbound}
Under Assumptions \ref{assump_strongconvex} and \ref{assump_lipgrad}, let the  sequence
$\{x^k\}_{k \ge 0}$ be generated by the algorithm SPP with
nonincreasing positive stepsizes $\{\mu_k\}_{k \ge 0}$. Then, the
following holds:
\begin{equation*}
 \mathbb{E}[\norm{\nabla f(x^k;S)}^2] \le 2\mathbb{E}[\norm{\nabla f(x^*;S)}^2] + 2\mathbb{E}[L_{f,S}^2] \mathcal{A}^2,
\end{equation*}
where $\mathcal{A} = \max\left\{\norm{x^0-x^*},\frac{\mu_0\sqrt{\mathbb{E}\left[\norm{\nabla f (x^*;S)}^2\right]}}{1 - \sqrt{\mathbb{E}\left[\theta_S^2(\mu_0)\right]} }\right\}.$
\end{lemma}

\begin{proof}
From the Lipschitz continuity of $\nabla f(\cdot;S)$ we have that $\norm{\nabla f(x;S) - \nabla f(x^*;S)} \le
L_{f,S}\norm{x-x^*}$ for all $x \in \rset^n$, which implies:
\begin{equation*}
 \norm{\nabla f(x^k;S)}^2 \le (\norm{\nabla f(x^*;S)} + L_{f,S} \norm{x^k - x^*})^2 \le 2\norm{\nabla f(x^*;S)}^2 + 2L_{f,S}^2 \norm{x^k - x^*}^2.
\end{equation*}
By taking expectation in both sides we get:
\begin{equation*}
 \mathbb{E}[\norm{\nabla f(x^k;S)}^2] \le 2\mathbb{E}[\norm{\nabla f(x^*;S)}^2] + 2\mathbb{E}[L_{f,S}^2] \mathbb{E}[\norm{x^k - x^*}^2].
\end{equation*}
Lastly, by using Lemma \ref{lemma_varstepbounded} we obtain our statement.
\end{proof}

\noindent We also present the following useful auxiliary result:

\begin{lemma}\label{lemma_main}
Let $\gamma \in (0,1]$ and the integers $p,q \in \mathbb{N}$ with $q
\ge p \ge 1 $.  Given the sequence of stepsizes $\mu_k =
\frac{\mu_0}{k^{\gamma}}$ for all  $k \ge 1$, where $\mu_0>0$, then
the following relation holds:
\begin{equation*}
\prod\limits_{i=p}^q \mathbb{E}[\theta_{S}^2(\mu_i)]  \le
\left(\mathbb{E}\left[\theta_S^2(\mu_0)
\right]\right)^{\varphi_{1-\gamma}(q+1) - \varphi_{1-\gamma}(p)}
\end{equation*}
\end{lemma}

\begin{proof}
From definition of $\theta_S(\mu)$ for any $k \ge 1$ we have:
\begin{align}
\mathbb{E}[\theta_{S}^2(\mu_k)]
& = \mathbb{E}\left[\left(\frac{1}{1 + \mu_k\sigma_{f,S}}\right)^2\right] =
\mathbb{E}\left[\frac{1}{\left(1 + \frac{\mu_0}{k^{\gamma}}\sigma_{f,S}\right)^2}\right] \nonumber \\
& \overset{\text{Lemma} \; \ref{lemma_bernoulli}}{\le}
\mathbb{E}\left[\left(\frac{1}{1 +
\mu_0\sigma_{f,S}}\right)^{\frac{2}{k^{\gamma}}}\right] \le
\left(\mathbb{E}\left[\frac{1}{(1 +
\mu_0\sigma_{f,S})^2}\right]\right)^{\frac{1}{k^{\gamma}}} =
\left(\mathbb{E}[\theta_{S}^2(\mu_0)]
\right)^{\frac{1}{k^{\gamma}}}. \label{upbound_alphak}
\end{align}
By taking into account that $\mathbb{E}[\theta_{S}^2(\mu_0)] =
\mathbb{E}\left[\frac{1}{(1 + \mu_0\sigma_{f,S})^2}\right] \le 1$
and that
$$\sum\limits_{i=p}^q \frac{1}{i^{\gamma}} \ge \varphi_{1-\gamma}(q+1) - \varphi_{1-\gamma}(p) = \int\limits_{p}^{q+1} \frac{1}{t^{\gamma}} dt =
\begin{cases}
 \ln{\frac{q+1}{p}} & \text{if} \quad \gamma =1 \\
\frac{(q+1)^{1-\gamma} - p^{1 -\gamma}}{1 - \gamma}  & \text{if} \quad \gamma < 1,
\end{cases}$$
then the relation \eqref{upbound_alphak} implies:
\begin{align}
\prod\limits_{i=p}^q \mathbb{E}[\theta_{S}^2(\mu_i)] & \le
\left(\mathbb{E}\left[\theta_S^2(\mu_0)\right]\right)^{\sum\limits_{i=p}^q
\frac{1}{i^{\gamma}}}  \le
\left(\mathbb{E}\left[\theta_S^2(\mu_0)\right]\right)^{\varphi_{1-\gamma}(q+1)
- \varphi_{1-\gamma}(p)} \nonumber \\
& =
\begin{cases}
\left(\mathbb{E}\left[\theta_S^2(\mu_0)\right]\right)^{\ln{\frac{q+1}{p}}} & \text{if} \quad \gamma =1 \\
\left(\mathbb{E}\left[\theta_S^2(\mu_0)\right]\right)^{\frac{(q+1)^{1-\gamma}
- p^{1-\gamma}}{1 - \gamma}} & \text{if} \quad \gamma < 1,
\end{cases}
\end{align}
which immediately implies the above statement.
\end{proof}

\noindent Finally, we provide a non-trivial upper  bound on the
feasibility gap, which automatically leads to a iterative descent in
the distance to the feasible set of the sequence $\{x^k\}_{k \ge
0}$, generated by the SPP scheme with nonincreasing stepsizes.

\begin{lemma}
\label{lemma_feas}
Under  Assumptions \ref{assump_feasset}, \ref{assump_strongconvex}
and \ref{assump_lipgrad}, let the  sequence $\{x^k\}_{k \ge 0}$ be
generated by SPP scheme with nonincreasing stepsizes $\{\mu_k\}_{k
\ge 0}$. Then, the following relation holds:
\begin{align*}
 \sqrt{\mathbb{E}[\text{dist}_X^2(x^{k})]}
\le \left(1-\frac{1}{\kappa}\right)^{k/2} \left[\text{dist}_X(x^0) + 2\mu_0\kappa \mathcal{B}\right]
+ 2\mu_{k-\lceil\frac{k}{2}\rceil}\kappa \mathcal{B},
\end{align*}
where $ \mathcal{B} = \sqrt{2\mathbb{E}[\norm{\nabla f (x^*;S)}^2]} + \mathcal{A} \sqrt{2\mathbb{E}[L^2_{f,S}]}$.
\end{lemma}

\begin{proof}
By using the strictly nonexpansive property of the projection operator \eqref{strict_nonexpansiveness} and the linear regularity assumption, we obtain:
\begin{align}
 \mathbb{E}[\text{dist}_X^2(x^{k+1})]
& \le \mathbb{E}[\norm{x^{k+1} - [y^k]_X}^2]
\le \mathbb{E}[\norm{y^k - [y^k]_X}^2] - \mathbb{E}[\norm{y^k - x^{k+1}}^2] \nonumber \\
& \overset{\text{As.} \; \ref{assump_feasset}}{\le} \mathbb{E}[\norm{y^k - [y^k]_X}^2] - \frac{1}{\kappa}\mathbb{E}[\norm{y^k - [y^k]_X}^2] \nonumber \\
& = \left(1 - \frac{1}{\kappa} \right) \mathbb{E}[\text{dist}^2_X(y^k)]. \label{lemma_feas_feasdecrease}
\end{align}
On the other hand, from triangle inequality and Minkowski inequality, we obtain:
\begin{align}
 \sqrt{\mathbb{E}[\text{dist}_X^2(y^k)]}
& \le \sqrt{\mathbb{E}[\norm{y^k - [x^{k}]_X}^2]}
 \le \sqrt{\mathbb{E}[(\norm{y^k - x^k} + \text{dist}_X(x^k))^2]}\nonumber  \\
& \overset{\text{Lemma} \; \ref{lemma_minkowski}}{\le} \sqrt{\mathbb{E}[\norm{z_{\mu_k}(x^k;S_k) - x^k}^2]} + \sqrt{\mathbb{E}[\text{dist}_X^2(x^k)]} \nonumber \\
& = \sqrt{\mathbb{E}[\text{dist}_X^2(x^k)]}  + \mu_k\sqrt{\mathbb{E}[\norm{\nabla f_{\mu_k}(x^k;S_k)}^2]}\nonumber  \\
& \overset{\text{Lemma} \; \ref{lemma_gradmiu_grad}}{\le} \sqrt{\mathbb{E}[\text{dist}_X^2(x^k)]}  + \mu_k\sqrt{\mathbb{E}[\norm{\nabla f (x^k;S_k)}^2]}\nonumber  \\
& \overset{\text{Lemma} \; \ref{lemma_gradbound}}{\le}
\sqrt{\mathbb{E}[\text{dist}_X^2(x^k)]}  +
\mu_k\left(\sqrt{2\mathbb{E}[\norm{\nabla f (x^*;S)}^2]} +  D
\sqrt{2\mathbb{E}[L^2_{f,S}]}\right).\label{lemma_feas_rel1}
\end{align}
For simplicity we use notations: $\alpha = \sqrt{1-
\frac{1}{\kappa}}, d_k = \sqrt{\mathbb{E}[\text{dist}_X^2(x^k)]}$
and $  \mathcal{B} = \sqrt{2\mathbb{E}[\norm{\nabla f (x^*;S)}^2]} +
\mathcal{A} \sqrt{2\mathbb{E}[L^2_{f,S}]}$. Combining
\eqref{lemma_feas_feasdecrease} and \eqref{lemma_feas_rel1} yields:
\begin{align}
d_{k+1} & \le \alpha d_k + \alpha \mu_k \mathcal{B}  \le
\alpha^{k+1} d_0 + \mathcal{B} \sum\limits_{i=1}^{k+1} \alpha^i
\mu_{k-i+1}. \label{lemma_feas_prelimrate}
\end{align}
Define $m = \lceil \frac{k+1}{2} \rceil$. By dividing the sum from the right side of \eqref{lemma_feas_prelimrate} in two parts and
by taking into account that $\{\mu_k\}_{k \ge 0}$ is nonincreasing, then results:
\begin{align*}
\sum\limits_{i=1}^{k+1} \alpha^i \mu_{k-i+1}
& = \sum\limits_{i=1}^{m} \alpha^i \mu_{k-i+1}  + \sum\limits_{i=m+1}^{k+1} \alpha^i \mu_{k-i+1} \nonumber \\
& \le \mu_{k-m+1}\sum\limits_{i=1}^{m} \alpha^i   + \alpha^{m+1}\sum\limits_{i=0}^{k-m} \alpha^i \mu_{k-i-m} \nonumber \\
& \le \mu_{k-m+1} \frac{\alpha(1 - \alpha^{m})}{1-\alpha}   + \mu_0\alpha^{m+1} \frac{1- \alpha^{k-m+1}}{1-\alpha} \nonumber \\
& \le \mu_{k-m+1} \frac{\alpha}{1-\alpha}   + \alpha^{m+1} \frac{\mu_0}{1-\alpha}.
\end{align*}
By using the last inequality into \eqref{lemma_feas_prelimrate} and using the bound $\frac{\alpha}{1-\alpha} \le 2\kappa $, then these
facts imply the statement of the lemma.
\end{proof}

\noindent Now,  we are ready to  derive the nonasymptotic
convergence rate of the Algorithm SPP with nonincreasing stepsizes.
For simplicity, we will use the following exponential approximation:
\begin{equation}\label{exponential_approx}
 e^{x} \ge
 1 + x \qquad \forall x \ge 0.
\end{equation}
We will also denote $\eta = \sqrt{\mathbb{E}[\norm{\nabla
f(x^*;S)}^2]}$ and keep the notations for $\mathcal{A}$ from Lemma
\ref{lemma_varstepbounded} and for $\mathcal{B}$ from Lemma
\ref{lemma_feas}.

\begin{theorem}\label{th_main_sc}
Under Assumptions \ref{assump_feasset}, \ref{assump_strongconvex} and
\ref{assump_lipgrad}, let the sequence $\{x^k\}_{k \ge 0}$ be
generated by the algorithm SPP with the stepsize  $\mu_k =
\frac{\mu_0}{k^\gamma}$ for all $k \geq 1$, with $\mu_0 > 0$ and
$\gamma \in (0,1]$, and denote $\theta_0 = \mathbb{E}\left[
\theta_{S}^2(\mu_0) \right]=  \mathbb{E}\left[\frac{1}{(1 + \mu_0
\sigma_{f,S})^2}\right]$. Then, the following relations hold:

\noindent $(i)$ If $\gamma \in (0,1)$, then we have the following
nonasymptotic convergence rates:
\begin{align*}
\mathbb{E}[\norm{ x^{k} - x^*}^2] & \!\le\!
\theta_0^{\varphi_{1-\gamma}(k)}r_0^2 +
\mathcal{D}\theta_0^{\varphi_{1-\gamma}(k) -
\varphi_{1-\gamma}(\frac{k+1}{2})}  \mu_0^2
\left[\varphi_{1-2\gamma}\left(\frac{k+1}{2} \right)+2\right]
 + \frac{\mathcal{D}\mu_0^2 4^{\gamma}}{(1-\theta_0)k^{\gamma}}.
\end{align*}

\noindent $(ii)$ If $\gamma = 1$,  then we have the following
nonasymptotic convergence rate:
\begin{align*}
\mathbb{E}[\norm{ x^{k} - x^*}^2] \le
\begin{cases}
\theta_0^{\varphi_{0}(k)}r_0^2 + \frac{2\mu_0^2}{k \left( \ln {\left(\frac{1}{\theta_0}\right)} - 1\right)}     & \text{if} \;\;\; \theta_0 < \frac{1}{e}\\
\theta_0^{\varphi_{0}(k)}r_0^2 +  \frac{2\mu_0^2\ln{k}}{k}     & \text{if} \;\;\; \theta_0 = \frac{1}{e}\\
\theta_0^{\varphi_{0}(k)}r_0^2 + \left( \frac{2}{k}\right)^{\ln
{\left(\frac{1}{\theta_0}\right)}} \frac{\mu_0^2}{1 - \ln
{\left(\frac{1}{\theta_0}\right)}}
    & \text{if} \;\;\; \theta_0 > \frac{1}{e},
\end{cases}
\end{align*}
where $\mathcal{D} = 4\norm{\nabla F(x^*)}\left[\frac{\text{dist}_X(x^0)+2\mu_0\kappa \mathcal{B}}{\mu_0\ln\left(\kappa/(\kappa -1)\right)}
+ 3^{\gamma}\mathcal{B}\kappa\right] +  2\eta \sqrt{2\eta^2 + 2\mathbb{E}[L_{f,S}^2]\mathcal{A}^2} + 2\eta \mathcal{A} \sqrt{\mathbb{E}[L_{f,S}^2]}$.
\end{theorem}

\begin{proof}
Let $\mu>0, x \in \rset^n$ and $S \in \Omega$, then we have:
\begin{align}
& \frac{1}{2}\norm{z_{\mu}(x;S) - x^*}^2 \nonumber \\
& = \frac{1}{2}\norm{z_{\mu}(x;S)-z_{\mu}(x^*;S)}^2 + \langle z_{\mu}(x;S) - z_{\mu}(x^*;S), z_{\mu}(x^*;S) -x^*\rangle + \frac{1}{2}\norm{z_{\mu}(x^*;S) -x^*}^2 \nonumber\\
& \le \frac{\theta_{S}^2(\mu)}{2}\norm{x-x^*}^2 - \mu\langle \nabla f(x^*;S), x -x^*\rangle + \langle z_{\mu}(x^*;S) -x^* + \mu\nabla f(x^*;S),x-x^* \rangle \nonumber \\
& \hspace{3cm} + \langle z_{\mu}(x;S) - x, z_{\mu}(x^*;S) -
x^*\rangle  -  \frac{\mu^2}{2}\norm{\nabla f_{\mu}(x^*;S)}^2.
\label{convrateroot_ineq}
\end{align}
Now we take expectation in both sides and consider $x=x^k$ and $ \mu
= \mu_k$. We thus seek a bound for each term from the right hand
side in \eqref{convrateroot_ineq}. For the second term, by using the
optimality conditions $\langle \nabla F(x^*),  z - x^* \rangle \ge
0$ for all $z \in X$, we have:
\begin{align}
\mathbb{E}[\langle & \nabla f(x^*;S),  x^* - x^k\rangle]
 = \mathbb{E}[\langle \nabla F(x^*), x^* - [x^k]_X \rangle] + \mathbb{E}[\langle \nabla F(x^*), [x^k]_X -x^k\rangle] \nonumber\\
& \le  \mathbb{E}[\langle \nabla F(x^*), [x^k]_X -x^k\rangle]  \nonumber\\
& \le \norm{\nabla F(x^*)} \;  \mathbb{E}[\text{dist}_X(x^k)] \le \norm{\nabla F(x^*)} \;  \sqrt{\mathbb{E}[\text{dist}_X^2(x^k)]} \nonumber \\
& \overset{\text{Lemma} \;\ref{lemma_feas}}{\le} \norm{\nabla
F(x^*)} \left[\left(1-\frac{1}{\kappa}\right)^{\frac{k}{2}} \left(\text{dist}_X(x^0) + 2\mu_0\kappa \mathcal{B}\right) + 2\mu_{k-\lceil \frac{k}{2}\rceil}\kappa\mathcal{B} \right]. \nonumber
\end{align}
By using relation \eqref{exponential_approx} and the fact that $\frac{1}{k} \le \frac{1}{k^{\gamma}} $ when $k \ge 1$ and $\gamma \in (0,1]$, then the last inequality implies:
\begin{align}
\mathbb{E}[\langle  \nabla f(x^*;S),  & x^* - x^k\rangle]
 \le \norm{\nabla F(x^*)} \left[\frac{2\text{dist}_X(x^0)+4\mu_0\kappa \mathcal{B}}{k \ln\left(\kappa/(\kappa -1)\right)}  + 2 \mu_{k-\lceil \frac{k}{2}\rceil}\mathcal{B}\kappa  \right]  \nonumber \\
& \le \mu_k\norm{\nabla F(x^*)} \left[\frac{2\text{dist}_X(x^0)+4\mu_0\kappa \mathcal{B}}{\mu_0\ln\left(\kappa/(\kappa -1)\right)}  + \frac{2 \mu_{k-\lceil \frac{k}{2}\rceil}\mathcal{B}\kappa}{\mu_k}  \right]. \label{boundproduct1}
\end{align}
For the third term in \eqref{convrateroot_ineq} we observe from the
optimality conditions for $z_{\mu_k}(x^*;S)$ that:
\begin{align*}
\left\|\frac{1}{\mu_k}\left(z_{\mu_k}(x^*;S) -x^* \right) + \nabla f (x^*;S)\right\|
& = \norm{\nabla f(z_{\mu_k}(x^*;S);S) - \nabla f(x^*;S)} \\
& \le L_{f,S} \norm{z_{\mu_k}(x^*;S) - x^*} = \mu_k L_{f,S} \norm{\nabla f_{\mu_k}(x^*;S)} \\
& \overset{\text{Lemma} \; \ref{lemma_gradmiu_grad}}{\le} \mu_k L_{f,S} \norm{\nabla f(x^*;S)},
\end{align*}
which yields the following bound:
\begin{align*}
 \langle z_{\mu_k}(x^*;S) -x^* + \mu_k\nabla f(x^*;S),x^k-x^* \rangle
& \le \norm{z_{\mu_k}(x^*;S) -x^* + \mu_k\nabla f(x^*;S)} \cdot \norm{x^k-x^*} \\
& \le \mu_k^2 L_{f,S} \norm{\nabla f(x^*;S)} \cdot \norm{x^k-x^*}.
\end{align*}
By taking expectation in both sides and using Lemma
\ref{lemma_varstepbounded}, we obtain the
refinement:
\begin{align}
\mathbb{E}[\langle z_{\mu_k}(x^*;S) -x^* + \mu_k\nabla f(x^*;S),x^k-x^* \rangle]
 & \le \mu_k^2 \sqrt{\mathbb{E}[L_{f,S}^2]} \sqrt{\mathbb{E}[\norm{\nabla f(x^*;S)}^2]}\mathbb{E}[\norm{x^k-x^*}] \nonumber\\
 & \overset{\text{Lemma} \; \ref{lemma_varstepbounded}}{\le} \mu_k^2 \sqrt{\mathbb{E}[L_{f,S}^2]}  \eta \mathcal{A}. \label{boundproduct2}
\end{align}
Finally, for the fourth term in \eqref{convrateroot_ineq} we use Lemma \ref{lemma_gradbound}:
\begin{align}
& \mathbb{E}[\langle z_{\mu_k}(x^k;S) - x^k, z_{\mu_k}(x^*;S) -
x^*\rangle] = \mu_k^2 \mathbb{E}[\norm{\nabla f_{\mu_k}(x^k;S)}\norm{\nabla f_{\mu_k}(x^*;S)}] \nonumber \\
& \overset{\text{Lemma} \; \ref{lemma_gradmiu_grad}}{\le} \mu_k^2 \mathbb{E}[\norm{\nabla f(x^k;S)}\norm{\nabla f(x^*;S)}]  \le \mu_k^2 \sqrt{\mathbb{E}[\norm{\nabla f(x^k;S)}^2]\mathbb{E}[\norm{\nabla f(x^*;S)}^2]} \nonumber  \\
& \overset{\text{Lemma} \; \ref{lemma_gradbound}}{\le} \mu_k^2 \eta \sqrt{2\eta^2 + 2 \mathbb{E}[L_{f,S}^2]\mathcal{A}^2} . \label{boundproduct3}
\end{align}
By taking expectation in \eqref{convrateroot_ineq}, using the
relations \eqref{boundproduct1}-\eqref{boundproduct3} and taking into account that $\frac{\mu_k}{\mu_{k-\lceil \frac{k}{2}\rceil}} \le 3^{\gamma}$ for all $k \ge 1$, we obtain:
\begin{align*}
& \mathbb{E}[\norm{z_{\mu_k}(x^k;S) - x^*}^2] \\
&  \le   \mathbb{E}\left[\theta_{S}^2(\mu_k) \norm{x^k - x^*}^2\right]   + 4\mu_k^2 \norm{F(x^*)}\left[\frac{\text{dist}_X(x^0)+2\mu_0\kappa \mathcal{B}}{\mu_0\ln\left(\kappa/(\kappa -1)\right)}
+ 3^{\gamma}\mathcal{B}\kappa  \right] \\
& \hspace{3cm} +  2\mu_k^2\eta \sqrt{2\eta^2 + 2\mathbb{E}[L_{f,S}^2]\mathcal{A}^2} + 2\mu_k^2\eta \mathcal{A} \sqrt{\mathbb{E}[L_{f,S}^2]} \\
& =   \mathbb{E}\left[\theta_{S}^2(\mu_k)\right]
\mathbb{E}[\norm{x^k - x^*}^2] + \mu_k^2\mathcal{D}.
\end{align*}
For simplicity, we use further in the proof the following notations:
$r_k = \sqrt{\mathbb{E}[\norm{x^k-x^*}^2]}$ and $\theta_k =
\mathbb{E}[\theta_{S}^2(\mu_k)]$. Then, through the nonexpansiveness
property of the projection operator, the previous inequality turns
into:
\begin{align}
r_{k+1}^2 & \le \mathbb{E}[\norm{z_{\mu_k}(x^k;S) - x^*}^2]  \le \theta_k r_k^2 + \mu_k^2
\mathcal{D} \nonumber \\
& \le r_0^2 \prod\limits_{i=0}^k \theta_i + \mathcal{D}
\sum\limits_{i=0}^k \left(\prod\limits_{j=i+1}^k \theta_j \right)
\mu_{i}^2. \label{upbound_rk}
\end{align}
To further refine the right hand side in \eqref{upbound_rk}, we
first notice   from Lemma \ref{lemma_main} that we have
$\prod\limits_{i=0}^k \theta_i \le
\theta_0^{\varphi_{1-\gamma}(k+1)}$. Then, from \eqref{upbound_rk}
we can  derive different upper bounds for the two cases of the
parameter $\gamma$:  $\gamma < 1$ and $\gamma = 1$.\\

\noindent  Case $(i)$ $\gamma < 1$.  From Lemma \ref{lemma_main}, we
derive an upper approximation for the second term in the right hand
side of \eqref{upbound_rk}. Therefore, if we let $m = \left\lceil
\frac{k}{2}\right\rceil$ we obtain:
\begin{align}
\sum\limits_{i=0}^k \mu_{i}^2 & \left(\prod\limits_{j=i+1}^k \theta_j\right)
= \sum\limits_{i=0}^m \mu_{i}^2\left(\prod\limits_{j=i+1}^k \theta_j\right)
+ \sum\limits_{i=m+1}^k \mu_{i}^2\left(\prod\limits_{j=i+1}^k \theta_j \right)  \nonumber\\
& \overset{\text{Lemma} \; \ref{lemma_main}}{\le}
\sum\limits_{i=0}^m \mu_{i}^2 \theta_0^{\varphi_{1-\gamma}(k+1) -
\varphi_{1-\gamma}(i+1)}  + \mu_{m+1}
\sum\limits_{i=m+1}^k \mu_{i} \left(\prod\limits_{j=i+1}^k \theta_j\right)  \nonumber\\
& \le \theta_0^{\varphi_{1-\gamma}(k+1) - \varphi_{1-\gamma}(m+1)}  \sum\limits_{i=0}^m \mu_{i}^2   + \mu_{m+1} \sum\limits_{i=m+1}^k \mu_{i} \left(\prod\limits_{j=i+1}^k \theta_j\right)  \nonumber\\
& = \theta_0^{\varphi_{1-\gamma}(k+1) - \varphi_{1-\gamma}(m+1)}
\sum\limits_{i=0}^m \mu_{i}^2   + \mu_{m+1} \sum\limits_{i=m+1}^k
\frac{\mu_{i}}{1-\theta_i} (1-\theta_i) \left(\prod\limits_{j=i+1}^k
\theta_j\right). \label{th_main_sc_prelimrate1}
\end{align}
We will further refine the right hand side  of
\eqref{th_main_sc_prelimrate1} by noticing the following two facts.
First, the constant $\frac{\mu_i}{1-\theta_i}$ can be upper bounded
by:
\begin{align*}
\frac{\mu_{i}}{1 - \theta_{i}} =
\frac{1}{\mathbb{E}\left[\frac{\sigma_S}{(1+ \mu_{i}\sigma_S)^2} +
\frac{\sigma_S}{1+ \mu_{i}\sigma_S} \right]} \le \frac{\mu_{i-1}}{1
- \theta_{i-1}} \le \cdots \le \frac{\mu_{0}}{1 - \theta_{0}}.
\end{align*}
Second, the sum of products is upper bounded as:
\begin{align*}
 \sum\limits_{i=m+1}^k (1-\theta_i) \left(\prod\limits_{j=i+1}^k \theta_j\right)
& = \sum\limits_{i=m+1}^k \left( \prod\limits_{j=i+1}^k \theta_j - \prod\limits_{j=i}^k \theta_j\right)  = 1 - \prod\limits_{j=m+1}^k \theta_j \le 1.
\end{align*}
By using the last two inequalities  into
\eqref{th_main_sc_prelimrate1}, we have:
\begin{align}
\sum\limits_{i=0}^k \mu_{i}^2\left(\prod\limits_{j=i+1}^k
\theta_j\right) & \le \theta_0^{\varphi_{1-\gamma}(k+1) -
\varphi_{1-\gamma}(m+1)}  \sum\limits_{i=0}^m \mu_{i}^2   +
\mu_{m+1} \frac{\mu_0}{1-\theta_0}. \label{th_main_sc_prelimrate2}
\end{align}
Since $\sum\limits_{i=0}^m \mu_{i}^2 \le \mu_0^2 (\varphi_{1-2\gamma}(m)+2) \le \mu_0^2 (\varphi_{1-2\gamma}(m)+2) \le \mu_0^2 [\varphi_{1-2\gamma}\left(\frac{k}{2} + 1 \right)+2] $ and  using
\eqref{th_main_sc_prelimrate2} into \eqref{upbound_rk}, we obtain
the above result.\\

\noindent  Case $(ii)$ $\gamma = 1$.  In this case  we have:
\begin{align*}
& \sum\limits_{i=1}^k \mu_{i}^2\left(\prod\limits_{j=i+1}^k
\theta_j\right) \overset{\text{Lemma} \; \ref{lemma_main}}{\le} \sum\limits_{i=1}^k \mu_{i}^2  \theta_0^{\varphi_{0}(k+1) - \varphi_0(i+1)} \nonumber \\
& = \sum\limits_{i=1}^k \frac{\mu_1^2}{i^2} \theta_0^{\ln
{\frac{k+1}{i+1}}} = \sum\limits_{i=1}^k \frac{\mu_1^2}{i^2}
\left(\frac{k+1}{i+1} \right)^{\ln { \theta_0  }}
\le \left( \frac{1}{k}\right)^{ \ln {\left(\frac{1}{\theta_0}\right)}} \sum\limits_{i=1}^k \frac{\mu_1^2}{i^{2 - \ln { \frac{1}{\theta_0}}}} \nonumber \\
& \le \left( \frac{1}{k}\right)^{\ln
{\left(\frac{1}{\theta_0}\right)}} \mu_0^2 \varphi_{\ln {
\frac{1}{\theta_0}} - 1}(k). \nonumber
\end{align*}
Therefore, the variation of $\theta_0$ leads to the following cases:
\begin{align*}
\sum\limits_{i=1}^k \mu_{i}^2\left(\prod\limits_{j=i+1}^k \theta_j\right) \le
\begin{cases}
\frac{\mu_0^2}{ k \left( \ln {\left(\frac{1}{\theta_0}\right)} - 1\right)}     & \text{if} \; \theta_0 < \frac{1}{e}\\
\frac{\mu_0^2\ln{k}}{k}     & \text{if} \; \theta_0 = \frac{1}{e}\\
\left( \frac{1}{k}\right)^{\ln {\left(\frac{1}{\theta_0}\right)}}
\frac{\mu_0^2}{1 - \ln {\left(\frac{1}{\theta_0}\right)}}     &
\text{if} \; \theta_0 > \frac{1}{e},
\end{cases}
\end{align*}
which leads to the second part of the result.
\end{proof}

\noindent For  more clear estimates of the  convergence rates
obtained in Theorem \ref{th_main_sc}, we provide in the next
corollary a summary given in terms of the dominant terms:

\begin{corrolary}\label{corr_main_sc}
Under the assumptions of Theorem \ref{th_main_sc} the following
convergence rates hold:\\
\noindent $(i)$ If $\gamma \in (0,1)$, then we have convergence rate
of order:
\begin{align*}
\mathbb{E}[\norm{ x^{k} - x^*}^2]
& \le \mathcal{O}\left(\frac{1}{k^{\gamma}}\right)
\end{align*}
\noindent $(ii)$ If $\gamma = 1$, then we have convergence rate of
order:
\begin{align*}
\mathbb{E}[\norm{ x^{k} - x^*}^2] \le
\begin{cases}
\mathcal{O}\left(\frac{1}{ k}\right)     & \text{if} \; \theta_0 < \frac{1}{e}\\
\mathcal{O}\left(\frac{\ln{k}}{ k}\right)     & \text{if} \; \theta_0 = \frac{1}{e}\\
\mathcal{O}\left(\frac{1}{k}\right)^{2 \ln
{\left(\frac{1}{\theta_0}\right)}}  & \text{if} \; \theta_0 >
\frac{1}{e}.
\end{cases}
\end{align*}
\end{corrolary}

\begin{proof}
First assume that $\gamma \in (0, \frac{1}{2})$. This assumption
implies that $1 - 2\gamma > 0$ and that:
\begin{align}\label{corr_main_sc_est1}
\varphi_{1-2\gamma}\left(\frac{k}{2}+2\right) = \frac{\left(\frac{k}{2} + 2 \right)^{1-2\gamma} - 1}{1-2\gamma}\le \frac{\left(\frac{k}{2} + 2 \right)^{1-2\gamma}}{1-2\gamma}.
\end{align}
On the other hand, by using the inequality $e^{-x} \le \frac{1}{1 + x}$ for all $x \in \rset^{}$, we obtain:
\begin{align*}
& \theta_0^{\varphi_{1-\gamma}(k+1) - \varphi_{1-\gamma}(\frac{k+1}{2})}
\varphi_{1-2\gamma}\left(\frac{k}{2} +2 \right)
=  e^{(\varphi_{1-\gamma}(k+1) - \varphi_{1-\gamma}(\frac{k+1}{2}))\ln {\theta_0}} \varphi_{1-2\gamma}\left(\frac{k}{2} + 2 \right) \\
&\le \frac{\varphi_{1-2\gamma}\left(\frac{k}{2} + 2 \right)}{1 +
[\varphi_{1-\gamma}(k+1) -
\varphi_{1-\gamma}(\frac{k}{2}+1)]\ln{\frac{1}{\theta_0}}}
\overset{\eqref{corr_main_sc_est1}}{\le} \frac{\frac{(k+4)^{1-2\gamma}}{2^{1-2\gamma} (1-2\gamma)}  }{\frac{1}{1-\gamma}[(k+1)^{1-\gamma} - (\frac{k}{2}+1)^{1-\gamma}]\ln{\frac{1}{\theta_0}}} \\
& = \frac{\frac{(k+4)^{1-2\gamma}}{2^{1-2\gamma }
(1-2\gamma)}}{\frac{(k+2)^{1-\gamma}}{1-\gamma}[(\frac{2}{3})^{1-\gamma}
- (\frac{1}{2})^{1-\gamma}]\ln{\frac{1}{\theta_0}}}
 = \frac{1-\gamma}{1-2\gamma}\frac{2^{\gamma}(k+4)^{-\gamma}}{[(\frac{2}{3})^{1-\gamma} - (\frac{1}{2})^{1-\gamma}]\ln{\frac{1}{\theta_0}}} \approx \mathcal{O}\left( \frac{1}{k^{\gamma}}\right).
\end{align*}
Therefore, in this case, the overall rate will be given by:
$$r_{k+1}^2 \le \theta_0^{\mathcal{O}(k^{1-\gamma})}r_0^2 + \mathcal{O}\left( \frac{1}{k^{\gamma}}\right) \approx\mathcal{O}\left( \frac{1}{k^{\gamma}}\right) .$$
If $\gamma = \frac{1}{2}$, then the definition of
$\varphi_{1-2\gamma}(\frac{k}{2}+2)$  provides that:
$$r_{k+1}^2 \le \theta_0^{\mathcal{O}(\sqrt{k})}r_0^2 + \theta_0^{\mathcal{O}(\sqrt{k})}\mathcal{O}(\ln{k}) +  \mathcal{O}\left( \frac{1}{\sqrt{k}}\right) \approx\mathcal{O}\left( \frac{1}{\sqrt{k}}\right) .$$
When $\gamma \in (\frac{1}{2}, 1)$, it is obvious that
$\varphi_{1-2\gamma}\left(\frac{k}{2}+2\right) \le \frac{1}{2\gamma
- 1}$ and therefore the order of the convergence rate  changes into:
$$r_{k+1}^2 \le \theta_0^{\mathcal{O}(k^{1-\gamma})}[r_0^2 + \mathcal{O}(1)] + \mathcal{O}\left( \frac{1}{k^{\gamma}}\right) \approx\mathcal{O}\left( \frac{1}{k^{\gamma}}\right).$$
Lastly, if $\gamma = 1$, by using $\theta_0^{\ln{k+1}} \le
\left(\frac{1}{k}\right)^{\ln{\frac{1}{\theta_0}}}$ we obtain the
second part of our result.
\end{proof}

\noindent Notice that the above results state that our SPP algorithm
with variable stepsize $\frac{\mu_0}{k^\gamma}$ converges with
$\mathcal{O}\left(\frac{1}{k^{\gamma}}\right)$ rate. Similar results
have been obtained in \cite{TouTra:16} for a particular objective
function of the form $f(A_S^T x)$ without any constraints and for
 $\gamma \in (1/2,1]$. Moreover, for $\gamma =1$ similar convergence rate, but in
asymptotic fashion and for unconstrained problems, has been derived
in \cite{RyuBoy:16}.  As we have already mentioned in the
introduction section, the convergence rate for the SGD scheme
contains an exponential term of the form $\frac{e^{C_2 \mu^2_0}}{k^{\alpha \mu_0}}$, which for a given iteration counter $k$ grows
exponentially in the initial stepsize $\mu_0$, see \cite{MouBac:11}.
Thus, although the SGD method achieves a  rate $O(\frac{1}{k})$ for
a variable stepsize $\frac{\mu_0}{k}$, if $\mu_0$ is chosen too
large, then it can induce catastrophic effects in the convergence
rate. However, one should notice that for our SPP method, Theorem
\ref{th_main_sc} does not contain this kind of exponential term,
therefore SPP is more robust than SGD scheme even in the constrained case. This can be also observed in numerical simulations, see Section \ref{sec_numerics} below.  Clearly, Corollary \ref{corr_main_sc} directly implies the following complexity estimates for attaining a suboptimal point $x^k$ satisfying $\mathbb{E}[\norm{ x^{k} - x^*}^2] \le \epsilon$.

\begin{corrolary}
\label{corr_main_sc_eps}
Under the assumptions of Theorem \ref{th_main_sc} and $\epsilon > 0$ the following estimates hold.
For $\gamma \in (0,1)$, if we perform:
\begin{align*}
\left\lceil \mathcal{O}\left(\frac{1}{\epsilon^{1/\gamma}}\right) \right\rceil
\end{align*}
iterations of SPP scheme with variable stepsize, then the sequence $\{x^k\}_{k \ge 0}$ satisfies $\mathbb{E}[\norm{ x^{k} - x^*}^2] \le \epsilon$.
Moreover, for $\gamma = 1$ and  $\theta_0 < \frac{1}{e}$, if we perform:
\begin{align*}
\left\lceil \mathcal{O}\left(\frac{1}{\epsilon}\right) \right\rceil
\end{align*}
iterations of SPP scheme with variable stepsize, then we have $\mathbb{E}[\norm{ x^{k} - x^*}^2] \le \epsilon$.
\end{corrolary}

\begin{proof}
 The proof follows immediately  from Corollary \ref{corr_main_sc}.
\end{proof}


\section{A restarted variant of Stochastic Proximal Point algorithm}
From previous section we easily notice  that an $\mathcal{O}\left( \frac{1}{\epsilon}\right)$ convergence rate is obtained
for the SPP algorithm with variable stepsize $\mu_k = \frac{\mu_0}{k}$
only when the initial stepsize $\mu_0$ is chosen sufficiently large such that $\theta_0 < \frac{1}{\sqrt{e}}$. However, this condition is not easy to check. Therefore, if $\mu_0$ is not chosen adequately, we can encounter the case  $\theta_0 > \frac{1}{\sqrt{e}}$, which leads to a convergence rate for the SPP scheme  of order
$\mathcal{O}\left(\epsilon^{-\frac{1}{2\ln{(1/\theta_0)}}}\right)$, that  is implicitly dependent on the choice of the initial stepsize $\mu_0$. In conclusion,
in order to remove this dependence on the initial stepsize of  the simple SPP scheme, we develop a restarting variant of it. This variant consists of running the SPP
algorithm (as a routine) for multiple times (epochs) and restarting it each time after a certain number of iterations. In each epoch $t$, the SPP scheme  runs for an estimated number of iterations $K_t$, which may vary over the epochs,
depending on the assumptions made on the objective function. More explicitly,  the Restarted Stochastic Proximal Point (RSPP) scheme has the following iteration:

\begin{center}
\framebox{
\parbox{13 cm}{
\begin{center}
\textbf{ Algorithm {\bf RSPP} }
\end{center}
{
Let $\mu_0 > 0$ and $x^{0,0} \in \rset^n$. For $t\geq 1$ do:}
\begin{enumerate}
\item Compute stepsize  $\mu_t$ and number of inner iterations  $K_t$
\item Set $x^{K_t,t}$ the average output of SPP$(x^{K_{t-1},t-1},\mu_{t})$ runned for $K_{t}$ iterations with constant stepsize $\mu_t$
\item If an outer stopping criterion is satisfied, then \textbf{STOP},
otherwise $t:=t+1$ and go to step $1$.
\end{enumerate}
}}
\end{center}

\noindent We  analyze below the nonasymptotic convergence rates of
the RSPP algorithm  under different assumptions  on the objective
function: first we assume Assumptions \ref{assump_strongconvex} and
\ref{assump_lipgrad} to hold for the objective functions, and then
we assume that the objective functions are polyhedral and thus
satisfy a  sharp minima like condition.


\subsection{Nonasymptotic sublinear convergence of algorithm RSPP}
\noindent In this section we analyze the convergence rate of the
sequence generated by the RSPP scheme, which repeatedly calls the
subroutine SPP with a constant stepsize, in multiple epochs. We
consider that SPP runs in epoch $t \ge 1$ with the constant stepsize
$\mu_t$ for $K_t$ iterations. As in  previous sections, we first
provide a descent lemma for the feasibility gap.
For simplicity, we keep the notations of $\mathcal{A}$ from Lemma \ref{lemma_gradbound}
and $\mathcal{B}$ from Lemma \ref{lemma_feas}.

\begin{lemma}
\label{lemma_feas_rspp}
Let Assumptions \ref{assump_feasset},
\ref{assump_strongconvex} and \ref{assump_lipgrad} hold. Also let
the  sequence $\{x^{K_t,t}\}_{t \ge 0}$ be generated by RSPP scheme
with nonincreasing stepsizes $\{\mu_t\}_{t \ge 0}$ and nondecreasing
epoch lengths $\{K_t\}_{t \ge 1}$ such that $K_t \ge 1$ for all $t
\ge 1$. Then, the following relation holds:
\begin{align*}
 \sqrt{\mathbb{E}[\text{dist}_X^2(x^{K_t,t})]} \le \left(1-\frac{1}{\kappa}\right)^{\sum_{i=1}^t \frac{K_i}{2}} \text{dist}_X(x^{0,0}) +
 2\left(1 - \frac{1}{\kappa}\right)^{\sum\limits_{i=t-\lceil\frac{t}{2}\rceil}^t \frac{K_i}{2}} \mu_0\kappa^2 \mathcal{B}
+ 2\mu_{t-\lceil\frac{t}{2}\rceil}\kappa^2 \mathcal{B}.
\end{align*}
\end{lemma}

\begin{proof}
The proof follows similar lines with the one of Lemma \ref{lemma_feas_feasdecrease}.
Therefore, by using notations: $\alpha = \sqrt{1- \frac{1}{\kappa}}$ and $d_{k,t} = \sqrt{\mathbb{E}[\text{dist}_X^2(x^{k,t})]}$ results:
\begin{align*}
 d_{k+1,t}
& \le \alpha d_{k,t} + \alpha \mu_t \mathcal{B}  \le \alpha^{k+1} d_{0,t} + \mu_t \mathcal{B} \sum\limits_{i=1}^{k+1} \alpha^i \\
& \le \alpha^{k+1} d_{0,t} + \mu_t \mathcal{B} \frac{\alpha}{1-\alpha}.
\end{align*}
By setting $k = K_t-1$, then the last inequality implies:
\begin{align*}
 d_{K_t,t}
 &\le \alpha^{K_t} d_{K_{t-1},t-1} + \mu_t \mathcal{B} \frac{\alpha}{1-\alpha}\\
&\le \alpha^{\sum_{i=1}^t K_i} d_{0,0} +
\mathcal{B} \frac{\alpha}{1-\alpha} \sum\limits_{j=0}^{t-1} \alpha^{\sum_{i=t-j+1}^t K_i} \mu_{t-j}.
\end{align*}
Now set $m = \lceil \frac{t}{2} \rceil$.
By dividing the sum from the right side of \eqref{lemma_feas_prelimrate} in two parts,
by taking into account that $\{\mu_t\}_{t \ge 0}$ is nonincreasing and $\{K_t\}_{t \ge 0}$ is nondecreasing,
then results:
\begin{align*}
\sum\limits_{j=0}^{t-1} \alpha^{\sum_{i=t-j+1}^t K_i} \mu_{t-j}
& = \sum\limits_{j=0}^{m} \alpha^{\sum_{i=t-j+1}^t K_i} \mu_{t-j}  + \sum\limits_{j=m+1}^{t-1} \alpha^{\sum_{i=t-j+1}^t K_i} \mu_{t-j} \nonumber \\
& \le \mu_{t-m} \sum\limits_{j=0}^{m} \alpha^{\sum_{i=t-j+1}^t K_i}   + \mu_0 \alpha^{K_t}\sum\limits_{j=m+1}^{t-1} \alpha^{\sum_{i=t-j+1}^{t-1} K_i} \nonumber \\
& \le \mu_{t-m} \frac{1 - \alpha^{m+1}}{1-\alpha}   + \mu_0\alpha^{\sum_{i=t-m}^t K_i} \frac{1- \alpha^{t-m+2}}{1-\alpha} \nonumber \\
& \le  \frac{\mu_{t-m}}{1-\alpha}   +  \frac{\mu_0\alpha^{\sum_{i=t-m}^t K_i}}{1-\alpha}. \nonumber
\end{align*}
By using the last inequality into \eqref{lemma_feas_prelimrate} and using the bound $\frac{\alpha}{1-\alpha} \le 2\kappa $, then these
facts imply the statement of the lemma.
\end{proof}

\noindent Next, we provide the non-asymptotic bounds on the iteration complexity of RSPP scheme.

\begin{theorem}\label{th_rspp_const}
Let Assumptions \ref{assump_feasset}, \ref{assump_strongconvex} and
\ref{assump_lipgrad} hold and $\epsilon, \mu_0 > 0$. Also let
$\gamma > 0$ and $\{x^{K_t,t}\}_{ t \ge 0}$ be generated by RSPP
scheme with $\mu_t = \frac{\mu_0}{t^{\gamma}}$ and $K_t = \lceil t^{\gamma} \rceil$. If we perform the following
number of epochs:
$$ T = \left\lceil \max\left\{ \ln{\left(\frac{2r_{0,0}^2}{\epsilon}\right)}\frac{1}{\ln{(1/\theta_0)}}, \left( \frac{2^{\gamma + 1} \mathcal{D}_r}{\epsilon}\mathcal{C}\right)^{1/\gamma}  \right\} \right \rceil,$$
then after a total number of SPP iterations of $\frac{T^{1+\gamma}}{1+\gamma}$, which is bounded by
\begin{align*}
\left\lceil \frac{1}{1 + \gamma}\max\left\{
\ln{\left(\frac{2r_{0,0}^2}{\epsilon}\right)}^{1+ \gamma}
\frac{1}{\ln{(1/\theta_0)}^{1+ \gamma}}, \left( \frac{2^{\gamma + 1}
\mathcal{D}_r}{\epsilon}\mathcal{C}\right)^{1 + \frac{1}{\gamma}}
\right\} \right\rceil,
\end{align*}
where $\mathcal{D}_r = 4\norm{\nabla
F(x^*)}\left[\frac{\text{dist}_X(x^{0,0})+2\mu_0\kappa^2
\mathcal{B}}{\mu_0\ln\left(\kappa/(\kappa -1)\right)} +
3^{\gamma}\mathcal{B}\kappa^2\right] +  2\eta \sqrt{2\eta^2 + 2\mathbb{E}[L_{f,S}^2]\mathcal{A}^2} + 2\eta \mathcal{A} \sqrt{\mathbb{E}[L_{f,S}^2]}$
and  $\mathcal{C} = \frac{1}{2(1-\gamma) \ln{1/\sqrt{\theta_0}}} +
\frac{\mu_1^2}{(1-\theta_0)^2}$,  we have
$\mathbb{E}[\norm{x^{K_T,T} - x^*}^2] \le \epsilon$.
\end{theorem}

\begin{proof}
First notice that from $e^x \ge 1+ x$ for all $x \ge 0$, we have
$\left(1-\frac{1}{\kappa}\right)^{\sum_{i=1}^t \frac{K_i}{2}} \le \left(1-\frac{1}{\kappa}\right)^{\frac{K_t}{2}} \le \frac{2}{K_t \ln{(\kappa/\kappa -1)}}$
and
$\left(1 - \frac{1}{\kappa}\right)^{\sum\limits_{i=t-\lceil\frac{t}{2}\rceil}^t \frac{K_i}{2}} \le \left(1-\frac{1}{\kappa}\right)^{\frac{K_t}{2}} \le \frac{2}{K_t \ln{(\kappa/\kappa -1)}}$,
which imply that Lemma 2 becomes
\begin{align}\label{th_main_rspp_feasrel1}
 \sqrt{\mathbb{E}[\text{dist}_X^2(x^{K_t,t})]} \le \mu_t \frac{2\text{dist}_X(x^{0,0})}{\mu_0 \ln{(\kappa/\kappa -1)}} +
 \mu_t \frac{4\kappa^2 \mathcal{B}}{\ln{(\kappa/\kappa -1)}} + 2\mu_{t-\lceil\frac{t}{2}\rceil}\kappa^2 \mathcal{B}.
\end{align}
It can be seen that by combining \eqref{th_main_rspp_feasrel1}
with a similar argument as in Theorem \ref{th_main_sc} we obtain a similar descent as \eqref{upbound_rk}.
Therefore, let $k \ge 0$ and $x^{k,t}$ be the $k$th iterate from the $t$th epoch.
Then, by denoting $r_{k,t}^2 =\mathbb{E}[\norm{x^{k,t}-x^*}^2]$, results:
\begin{align*}
 r_{k+1,t}^2
& \le \mathbb{E}[\theta_{S}(\mu_{t})^2] r_{k,t}^2 + \mu_t^2 \mathcal{D}_r.
\end{align*}
Now taking $k = K_t$ results in:
\begin{align}
 r_{0,t+1}^2 = r_{K_t,t}^2
& \le r_{0,t}^2 \theta_t^{K_t} + \mathcal{D}_r  \mu_{t}^2 \sum\limits_{i=0}^k \theta_t^i  \le r_{0,t}^2 \theta_t^{K_t} + \frac{\mathcal{D}_r  \mu_{t}^2}{1-\theta_t}. \label{th_rspp_const_reccurence1}
\end{align}
Recalling that we chose $\mu_t = \frac{\mu_0}{t^{\gamma}}$ and $K_t = \lceil t^{\gamma} \rceil$, then Lemma \ref{lemma_bernoulli}  leads to:
\[ \theta_t^{K_t} \le \left(\mathbb{E}\left[ \frac{1}{(1 + \mu_0 \sigma_{f,S})^2}\right]\right)^{\frac{K_t}{t^{\gamma}}} \le \theta_0.  \]
Therefore, \eqref{th_rspp_const_reccurence1} leads to:
\begin{align}
r_{0,t+1}
& \overset{\eqref{th_rspp_const_reccurence1}}{\le} \theta_0 r_{0,t}^2 + \frac{\mathcal{D}_r  \mu_{t}^2}{1-\theta_t}  \le \theta_0^t r_{0,1}^2 + \mathcal{D}_r
 \sum\limits_{i=1}^t \frac{\mu_{i}^2 \theta_0^{t-i}}{1-\theta_i}.
\label{th_rspp_const_rel1}
\end{align}
Note that $\frac{\mu_{i}^2}{1-\theta_i}$ is nonincreasing in $i$. Then, if we  fix $m = \lceil \frac{t}{2}\rceil$, then the sum $
\sum\limits_{i=1}^t \frac{\mu_{i}^2 \theta_0^{t-i}}{1-\theta_i}$ can
be bounded as follows:
\begin{align}
 \sum\limits_{i=1}^t \frac{\mu_{i}^2 \theta_0^{t-i}}{1-\theta_i}
 & \le  \theta_0^m \sum\limits_{i=1}^m \frac{\mu_{i}^2}{1-\theta_i} + \sum\limits_{i=m}^t \frac{\mu_{i}^2 \theta_0^{t-i}}{1-\theta_i} \nonumber \\
 & \le  \theta_0^m \sum\limits_{i=1}^m \frac{\mu_{i}^2}{1-\theta_i} + \frac{\mu_{m}^2}{1-\theta_m} \sum\limits_{i=1}^{t-m} \theta_0^{i} \nonumber \\
 & \le  \frac{\theta_0^m\mu_{1}}{1-\theta_0} \left(\sum\limits_{i=1}^m \mu_i\right)  + \frac{\mu_{m}^2}{(1-\theta_m) (1- \theta_0)}  \nonumber \\
 & \le   \frac{\theta_0^m\mu_{1}}{1-\theta_0} \left(\sum\limits_{i=1}^m \mu_i\right)  + \mu_{m}\frac{\mu_{1}}{(1- \theta_0)^2}.  \label{th_rspp_const_rel2}
\end{align}
Taking into account that $\sum\limits_{i=1}^m \mu_i \le \int_1^m
\frac{1}{s^{\gamma}} ds \le \frac{2^{\gamma
-1}}{(1-\gamma)t^{\gamma-1}}$ and that $\theta_0^m \le \frac{1}{1 +
\frac{t}{2}\ln{\frac{1}{\theta_0}}}$, the previous relation
\eqref{th_rspp_const_rel2} implies:
\begin{align}
 \sum\limits_{i=1}^t \frac{\mu_{i}^2 \theta_0^{t-i}}{1-\theta_i} \le \left(\frac{2}{t}\right)^{\gamma}\left[\frac{1}{2(1-\gamma) \ln{1/\sqrt{\theta_0}}} + \frac{\mu_1^2}{(1-\theta_0)^2} \right].
\end{align}
By using this bound in relation \eqref{th_rspp_const_rel2}, then in order to obtain $r_{0,t+1}^2 \le \epsilon$ it is sufficient that the number of epochs $t$ to satisfy:
\begin{align}
 t \ge  \max\left\{ \ln{\left(\frac{2r_{0,0}^2}{\epsilon}\right)}\frac{1}{\ln{(1/\theta_0)}}, \left( \frac{2^{\gamma + 1} \mathcal{D}_r}{\epsilon}\mathcal{C}\right)^{1/\gamma}  \right\}.
\end{align}
Finally, the total number of SPP iterations performed by RSPP algorithm satisfies:
\begin{align*}
\sum\limits_{i=1}^t K_t & \ge \sum\limits_{i=1}^t i^\gamma
\ge \int_0^t s^{\gamma} ds = \frac{t^{1 + \gamma}}{1 + \gamma} \\
&\ge \frac{1}{1 + \gamma}\max\left\{\ln{\left(\frac{2r_{0,0}^2}{\epsilon}\right)}^{1+ \gamma} \frac{1}{\ln{(1/\theta_0)}^{1+ \gamma}}, \left( \frac{2^{\gamma + 1}
\mathcal{D}_r}{\epsilon}\mathcal{C}\right)^{1 + \frac{1}{\gamma}} \right\},
\end{align*}
which proves the statement of the theorem.
\end{proof}

\noindent In conclusion Theorem \ref{th_rspp_const} states that the RSPP algorithm
with the choices $(\mu_t, K_t) = \left(\frac{\mu_0}{t^{\gamma}},\frac{t^{\gamma}}{2}\right)$ requires $\mathcal{O}\left( \epsilon^{-\left(1 + \frac{1}{\gamma}\right)}\right)$ simple SPP iterations to reach an
$\epsilon$ optimal point. It is important to observe that this convergence rate is achieved when the stepsize and the epoch length are not dependent on any inaccessible constant,
making our restarting scheme easily implementable. Moreover, the parameter $\gamma$ can be chosen in $(0, \infty)$, i.e. our RSPP scheme allows also stepsizes $\frac{\mu_0}{t^{\gamma}}$, with $\gamma>1$.  By comparison, an $\mathcal{O}\left(\epsilon^{-1}\right)$ complexity is obtained for SPP with stepsize $\mu_k = \frac{\mu_0}{k}$ only when $\mu_0$ is
chosen sufficiently large such that $\theta_0 < \frac{1}{\sqrt{e}}$. However,  this condition is not easy to check. Moreover,  we  may fall in
the case when $\theta_0 > \frac{1}{\sqrt{e}}$, which leads to a
complexity of $\mathcal{O}\left(\epsilon^{-\frac{1}{2\ln{(1/\theta_0)}}}\right)$
of the variable stepsize  SPP scheme. Observe that the last convergence
rate is implicitly dependent on the constant $\mu_0$ and can be arbitrarily bad, while for $\gamma>1$ sufficiently large the RSPP  scheme achieves the optimal convergence rate $\mathcal{O}\left( \epsilon^{-1}\right)$.


\section{Contributions in the light of prior work} Notice that the
above results (see Theorem \ref{th_main_sc}) state that our SPP
algorithm with variable stepsize $\mu_k = \mu_0/k^\gamma$ and
$\gamma \in (0, 1]$ converges with
$\mathcal{O}\left(1/k^{\gamma}\right)$ rate for strongly convex
smooth constrained  optimization problems. When the objective
function is strongly convex and smooth and with a proper learning
rate for $\gamma=1$, our algorithm converges with $\mathcal{O}(1/k)$
rate, which is optimal for the stochastic methods for this problem
class. As we have already discussed in the introduction section, the
convergence rate for the SGD scheme contains an exponential term of
the form $e^{C \mu^2_0}/k^{\alpha \mu_0}$, which for a given
iteration counter $k$ grows exponentially in the initial stepsize
$\mu_0$, see \cite{MouBac:11}. Thus, although the SGD method
achieves a  rate $O({1}/{k})$ for a variable stepsize $\mu_k =
{\mu_0}/{k}$, if $\mu_0$ is chosen too large, then it can induce
catastrophic effects in the convergence rate. However, one should
notice that for our SPP method, Theorem \ref{th_main_sc} does not
contain this kind of exponential term, therefore SPP is more robust
than SGD scheme even in the constrained case. This can be also
observed in numerical simulations, see Section \ref{sec_numerics}
below.

\noindent Similar convergence results for SPP have been obtained in
\cite{TouTra:16} for a particular objective function of the form
$f(x;S)= \ell(A_S^T x)$ without any constraints and for $\gamma \in
(1/2,1]$. However, the analysis in \cite{TouTra:16} cannot be
trivially extended to the general convex objective functions and
complicated constraints, since for the proofs  it is essential that
each component of the objective function has the form $f(A^T_S x)$.
In our paper we consider general convex objective functions which
lack the previously discussed structure and also (in)finite number
of convex constraints. Moreover, for $\gamma =1$ similar convergence
rate, but in \textit{asymptotic} fashion and for unconstrained
problems, has been derived in \cite{RyuBoy:16}. Another paper
related to our work is \cite{Bia:16}, where the author also proposes
a SPP-like algorithm and  asymptotic convergence is established
without rates. To the best of our knowledge, our SPP method is the
first stochastic proximal point algorithm that can tackle
optimization problems with complicated constraints
\eqref{convexfeas_random}. Moreover, the convergence analysis is
non-trivial  and does not follow from the analysis corresponding to
the unconstrained settings.

\section{Numerical experiments}
\label{sec_numerics} We present numerical evidence to assess the
theoretical convergence guarantees of the SPP algorithm. We provide
two numerical examples:  constrained stochastic least-square with
random generated data \cite{MouBac:11,TouTra:16} and  Markowitz
portfolio optimization  using real data \cite{BroDau:09,YurVu:16}.
In all our figures the results are averaged over $30$ Monte-Carlo
simulations for an algorithm.


\subsection{Stochastic least-square problems using random data}
In this section we evaluate the practical performance of the SPP
schemes on finite large scale least-squares models. To do so, we
follow a simple normal (constrained) linear regression example from
\cite{MouBac:11,TouTra:16}. Let $m = 10^5$ be the number of
observations, and $n = 20$ be the number of features. Let $x^*$ be a
randomly a priori chosen ground truth. The random variable $S$ is
decomposed as $S = (a_i , b_i)$, where the feature vectors $a_1
,\cdots, a_m \approx \mathcal{N}_n (0, H)$ are i.i.d. normal random
variables, and $H$ is a randomly generated symmetric matrix with
eigenvalues $1/k$, for $k = 1, \cdots, n$. The outcome $b_i$ is
sampled from a normal distribution as $b_i | a_i \approx \mathcal{N}
(a_i^T x^*, 1)$, for $i = 1, \cdots, m$. Since the typical loss
function is defined as the elementary squared residual $(a_i^Tx -
b_i)^2$, which is not strongly convex, we consider batches of
residuals to form our loss functions, i.e we  consider $\ell(x, i )$
of two forms:
\[  \ell(x, i ) = \norm{A_{j(i):j(i)+n} x - b_{j(i):j(i)+n}}^2 \quad \text{or} \quad
 \ell(x, i ) = (a_i^Tx - b_i)^2, \]
where $a_i$ is the $i$th row of $A$ and  $A_{j(i):j(i)+n} \in
\rset^{n \times n}$ is a submatrix containing  $n$ rows of $A$ so
that the function $x \mapsto \norm{A_{j(i):j(i)+n} x -
b_{j(i):j(i)+n}}^2$ is strongly convex. In our tests we    used
${\verb round } \left(m/2n\right)$ batches of dimension $n$ and we
let the rest  as elementary residuals, thus having in total
$p=m/2+m/n$ loss functions. Additionally, we impose on the estimator
$x$ also $p$ linear inequality constraints $\{x \; | \; Cx \le d
\}$. This constraints can be found in many applications and they
come from physical constraints, see e.g. \cite{CenChe:12,RosVil:14}.
We choose randomly the matrix $C$ for the constraints and $d=C\cdot
x^* + [0 \; 0 \; 0 \; v^T]^T$, where $v \ge 0$ is a random vector of
appropriate dimension, i.e. three inequalities are active at the
solution $x^*$. Besides the SPP and RSPP algorithms analyzed in the
previous sections of our paper, we also implemented SGD and the
averaged variant of SPP algorithm (A-SPP), which has the same SPP
iteration, but outputs the average of iterates:  $\hat{x}^k =
({1}/{\sum_{i=1}^k \mu_i}) \sum_{i=1}^k \mu_ix_i$.

\noindent In Figure 1 we run algorithms SPP, RSPP, A-SPP and SGD for
two values of the initial stepsize: $\mu_0 = 0.5$ and $\mu_0 = 1$.
Each scheme  runs  for two stepsize exponents: $\gamma_1 = 1$ (left)
and $\gamma_2 = {1}/{2}$ (right). From Figure 1  we can asses one
conclusion of Theorem 17: that the best performance for SPP is
achieved for  stepsize exponent $\gamma = 1$. Moreover, we can
observe that algorithm RSPP has the fastest behavior, while the
averaged variant A-SPP is more robust to changes in the initial
stepsize $\mu_0$. The performance of SGD is much worse as exponent
$\gamma$ decreases and it is also sensitive to the learning rate
$\mu_0$. Notice that both tests are performed over $m$ iterations
(i.e. one pass through data).
\begin{figure}[h]\label{fig_all_comparison}
\begin{center}
\includegraphics[width=7.5cm,height=5cm]{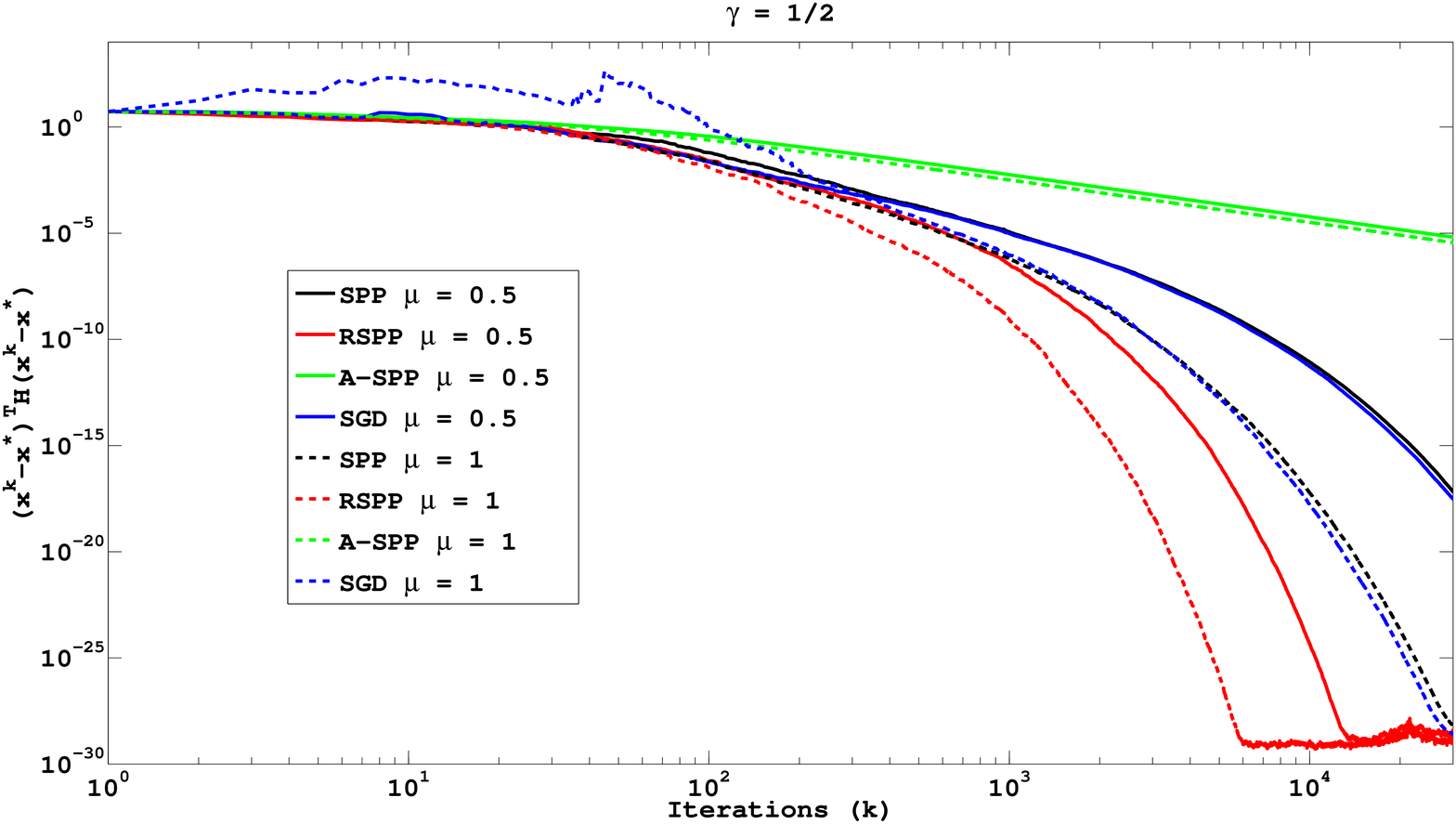}
\includegraphics[width=7.5cm,height=5cm]{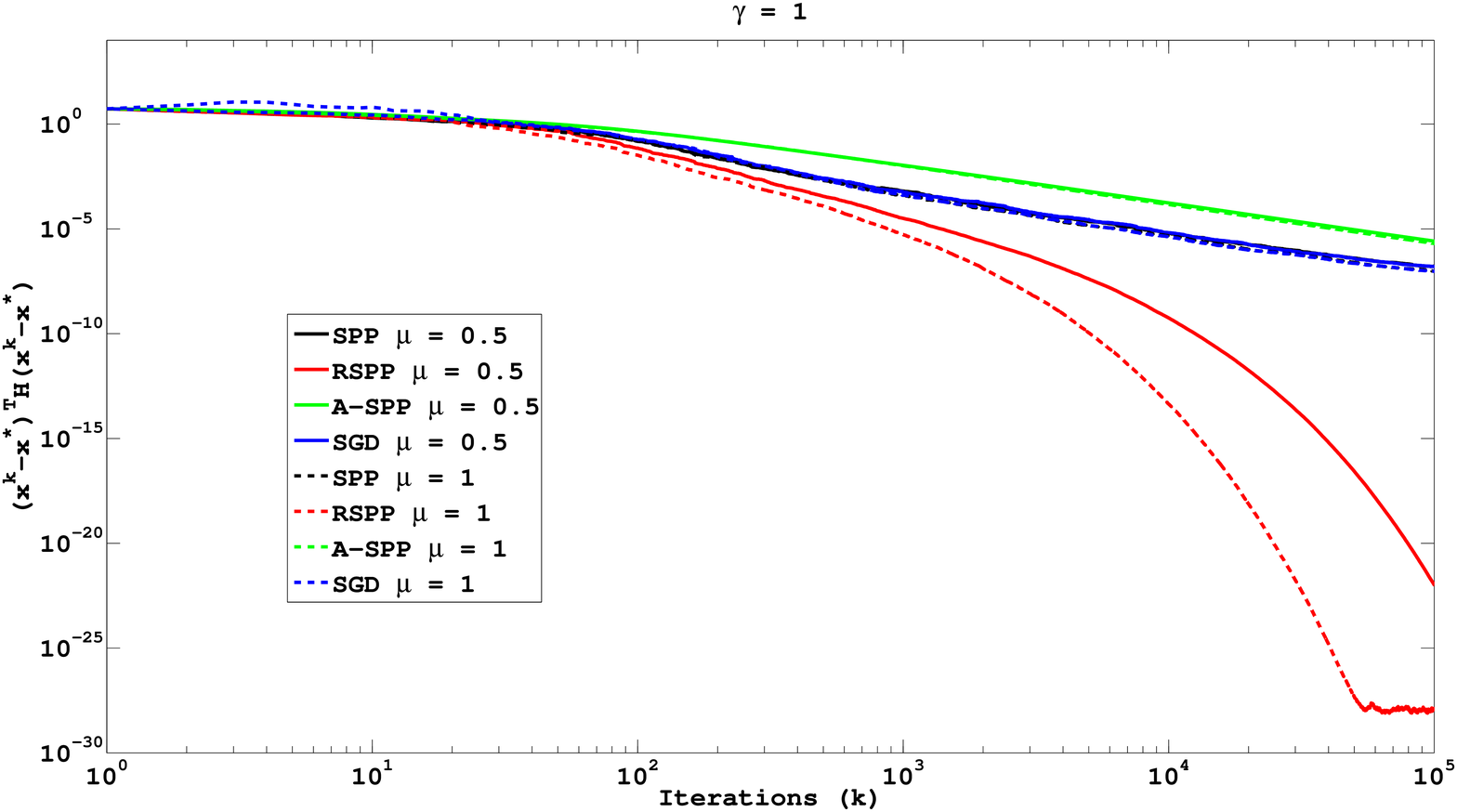}
\end{center}
\caption{Performance comparison of SPP, A-SPP, RSPP and SGD  for two
values of initial stepsize $\mu_0 \!=\! 0.5$ and $\mu_0 \!=\! 1$ and
for two values of exponent $\gamma\!=\!1/2$ (left) and
$\gamma\!=\!1$ (right).}
\end{figure}

\vspace{0.1cm}

\noindent  \noindent In the second set of experiments, we generate
random least-square  problems of the form $\min_{x: Cx \le d} \;
{1}/{2}\norm{Ax-b}^2$, where both matrices $A$ and $C$ have $m=10^3$
rows and generated randomly. Now, we do not impose the solution
$x^*$ to have the form given in the first test. We let SPP and RSPP
algorithms to do one pass through data for various stepsize
exponents $\gamma$. From Figure 2 we can assess the empirical
evidence of the ${\cal O}(1/\epsilon^{1/\gamma})$ convergence rate
of Theorem 17 for SPP and ${\cal O}(1/\epsilon^{1+ 1/\gamma})$
convergence rate of Theorem 21 for RSPP, by presenting squared
relative distance to the optimum solution. Moreover, the simulation
results match another conclusion of Theorems 17 and 21  regarding
the stepsize exponent $\gamma$: which state that the performance of
SPP/RSPP deteriorates with the decrease in the value of the stepsize
exponent.
\begin{figure}[ht]
\label{fig_comp_exponent}
\begin{center}
\includegraphics[width=7.5cm,height=5cm]{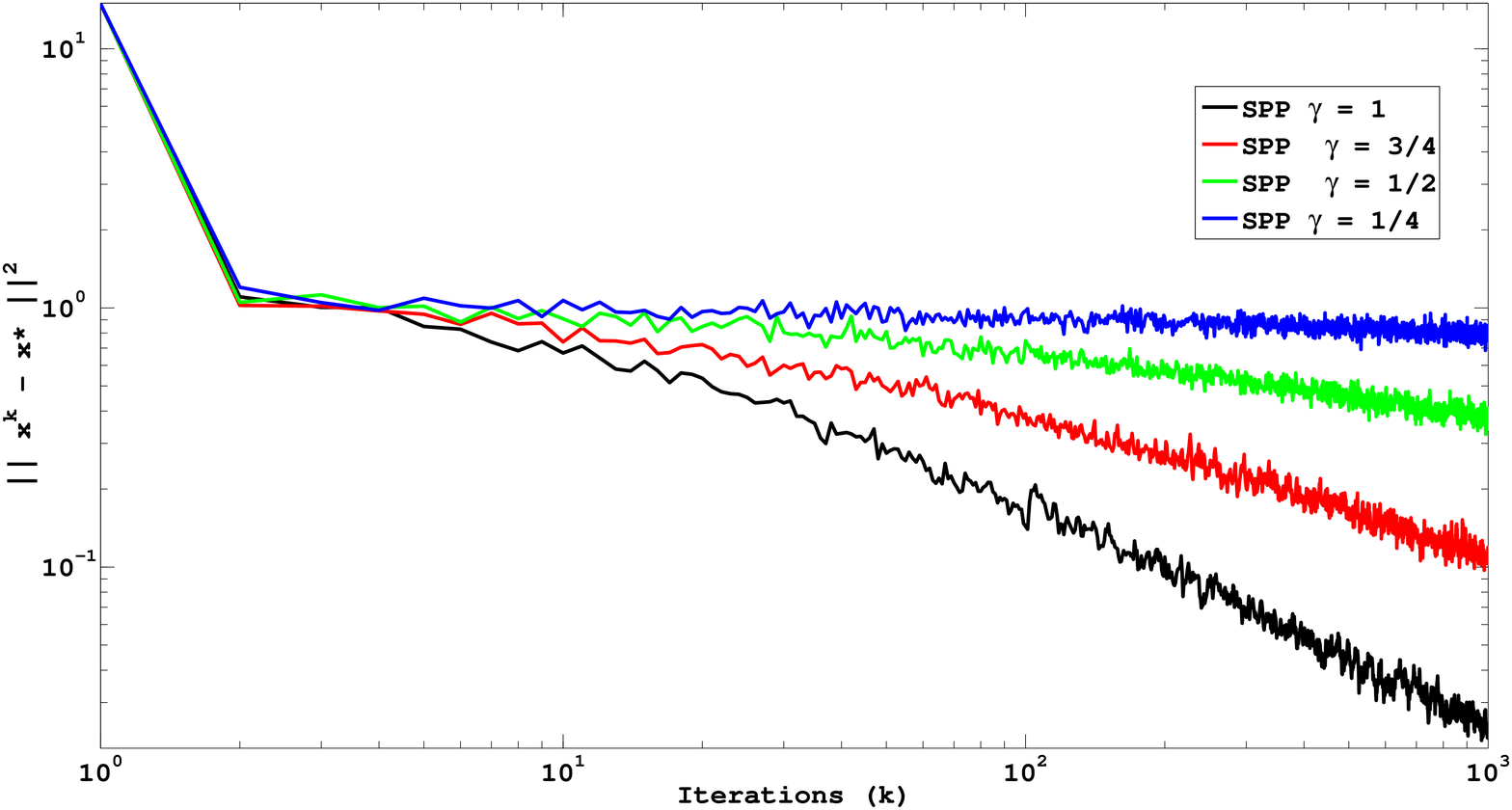}
\includegraphics[width=7.5cm,height=5cm]{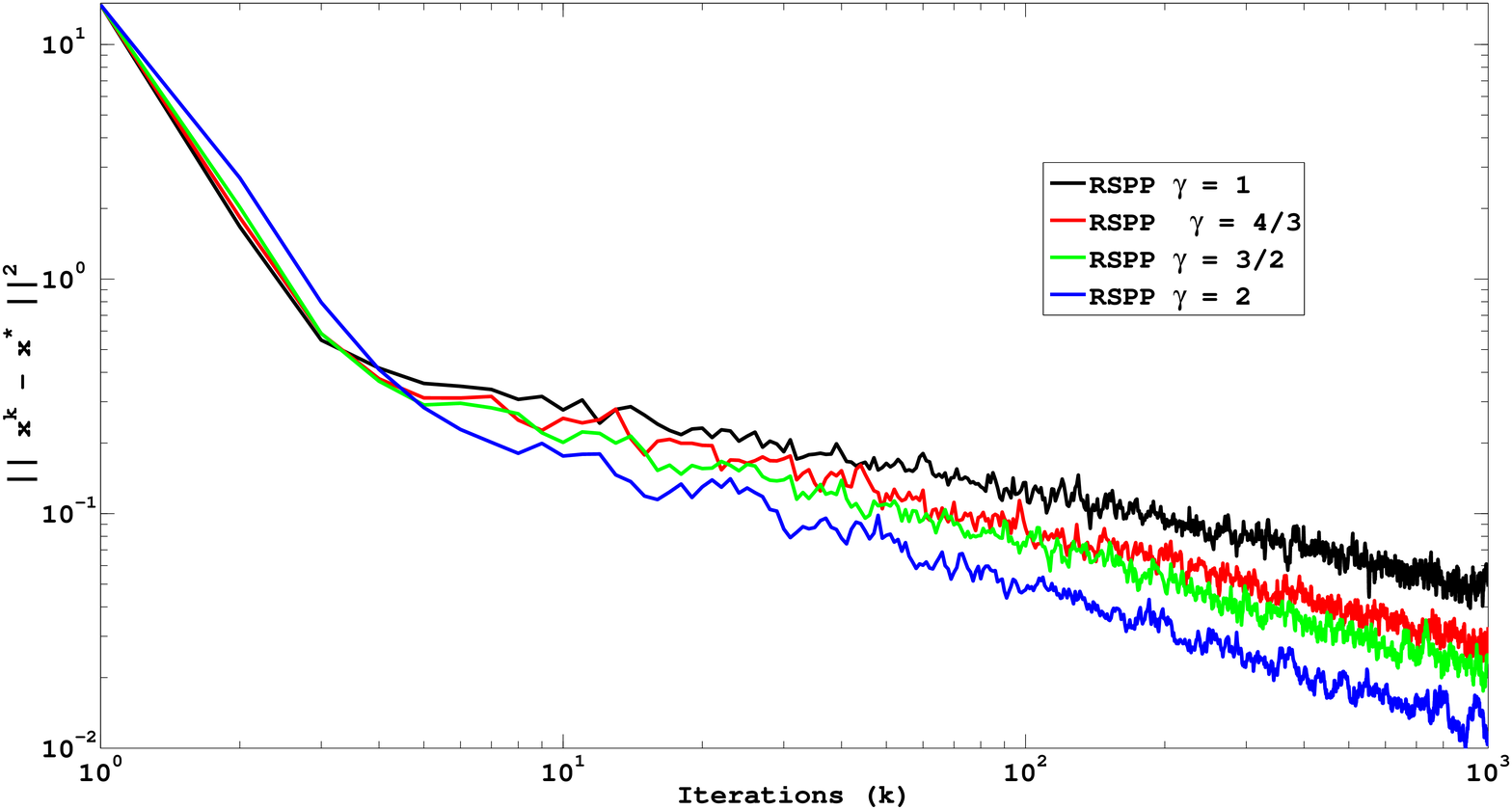}
\end{center}
\caption{Performance  of:  SPP  for four values of  the stepsize
exponent $\gamma=1, \; 3/4, \;  1/2$ and $1/4$ (left); RSPP  for
four values of  the stepsize exponent $\gamma=1, \; 4/3, \;  3/2$
and $2$ (right).}
\end{figure}


\subsection{Markowitz portfolio optimization  using real data}
Markowitz portfolio optimization aims to reduce the risk by
minimizing the variance for a given expected return. This can be
mathematically formulated as a convex optimization problem
\cite{BroDau:09,YurVu:16}:
\[ \min_{x \in \rset^n} \mathbb{E}[(a_S^Tx - b)^2]  \quad \text{s.t.} \quad
x \in X=\{x: \; x \geq 0, \; e^Tx \leq 1, \; a_{av}^T x \geq b \},
\] where $a_{av} = \mathbb{E}[a_S]$ is the average returns for each
asset that is assumed to be known (or estimated), and $b$ represents
a minimum desired return. Since new data points are arriving
on-line, one cannot access the entire dataset at any moment of time,
which makes the stochastic setting more favorable. For simulations,
we approximate the expectation with the empirical mean as follows:
\[ \min_{x \in \rset^n} \frac{1}{m} \sum_{i=1}^m (a_i^Tx - b)^2  \quad \text{s.t.} \quad
x \in X=X_1 \cap X_2 \cap X_3, \] where $X_1= \{x: \; x \geq 0\}$,
$X_2=\{x: \; e^Tx \leq 1 \}$ and $X_3= \{x: \; a_{av}^T x \geq b
\}$.  We use 2 different real portfolio datasets:  Standard \&
Poor's 500 (SP500, with 25 stocks for 1276 days) and one dataset by
Fama and French (FF100, with 100 portfolios for 23.647 days) that is
commonly used in financial literature, see e.g. \cite{BroDau:09}. We
split all the datasets into test ($10\%$) and train ($90\%$)
partitions randomly. We set the desired return $a_{av}$ as the
average return over all assets in the training set and $b =
\text{mean}(a_{av})$. The results of this experiment are presented
in Figure 3. We plot the value of the objective function over the
datapoints in the test partition $F_\text{test}$ along the
iterations. We observe that SGD is very sensitive to both
parameters, initial stepsize ($\mu_0$) and stepsize exponent
($\gamma$), while SPP is more robust to changes in both parameters
and also performs better over one pass through data in the train
partition.
\begin{figure}[h]
\label{fig_all_comparison}
\begin{center}
\includegraphics[width=6.5cm,height=5cm]{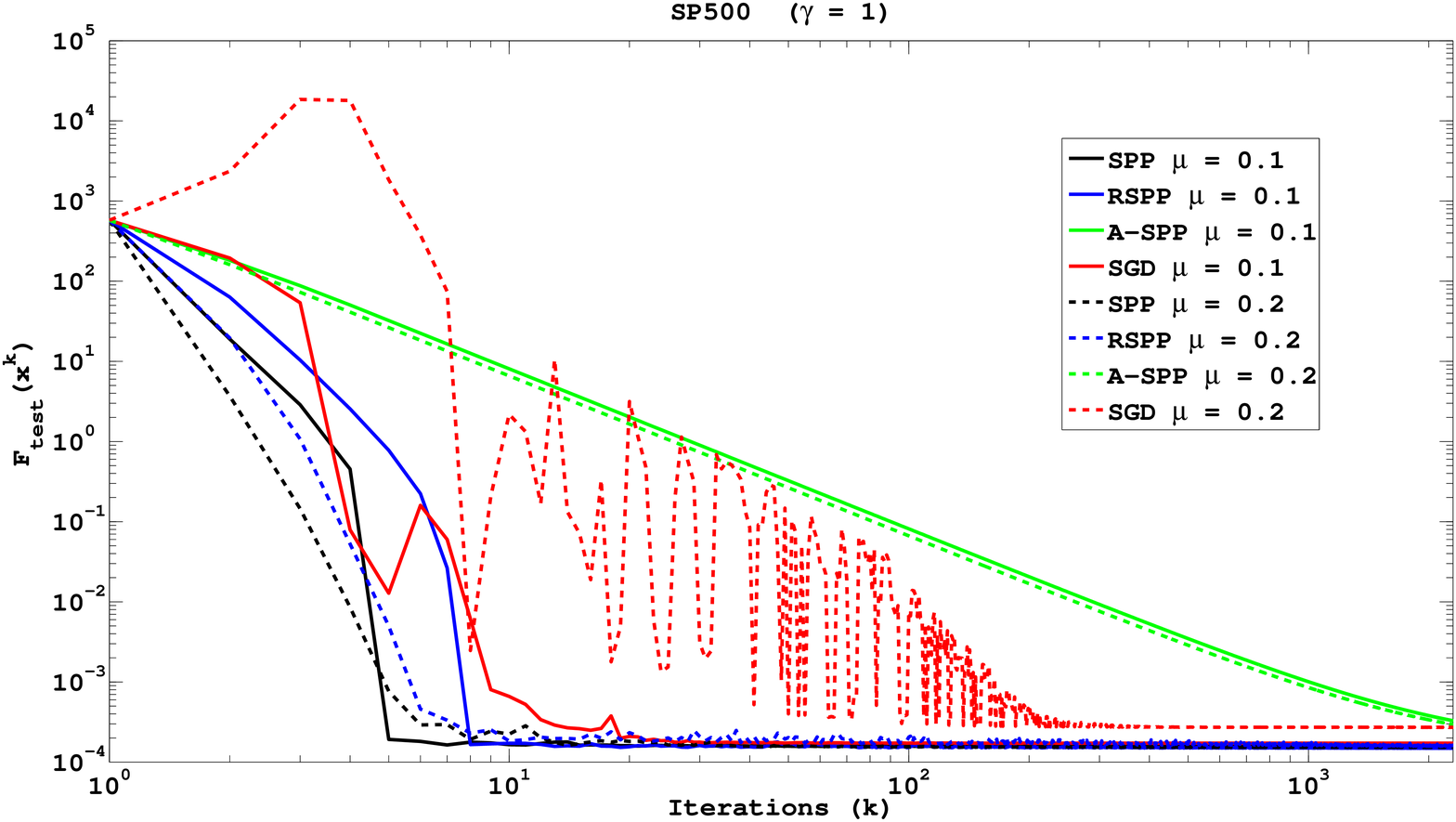}
\includegraphics[width=6.5cm,height=5cm]{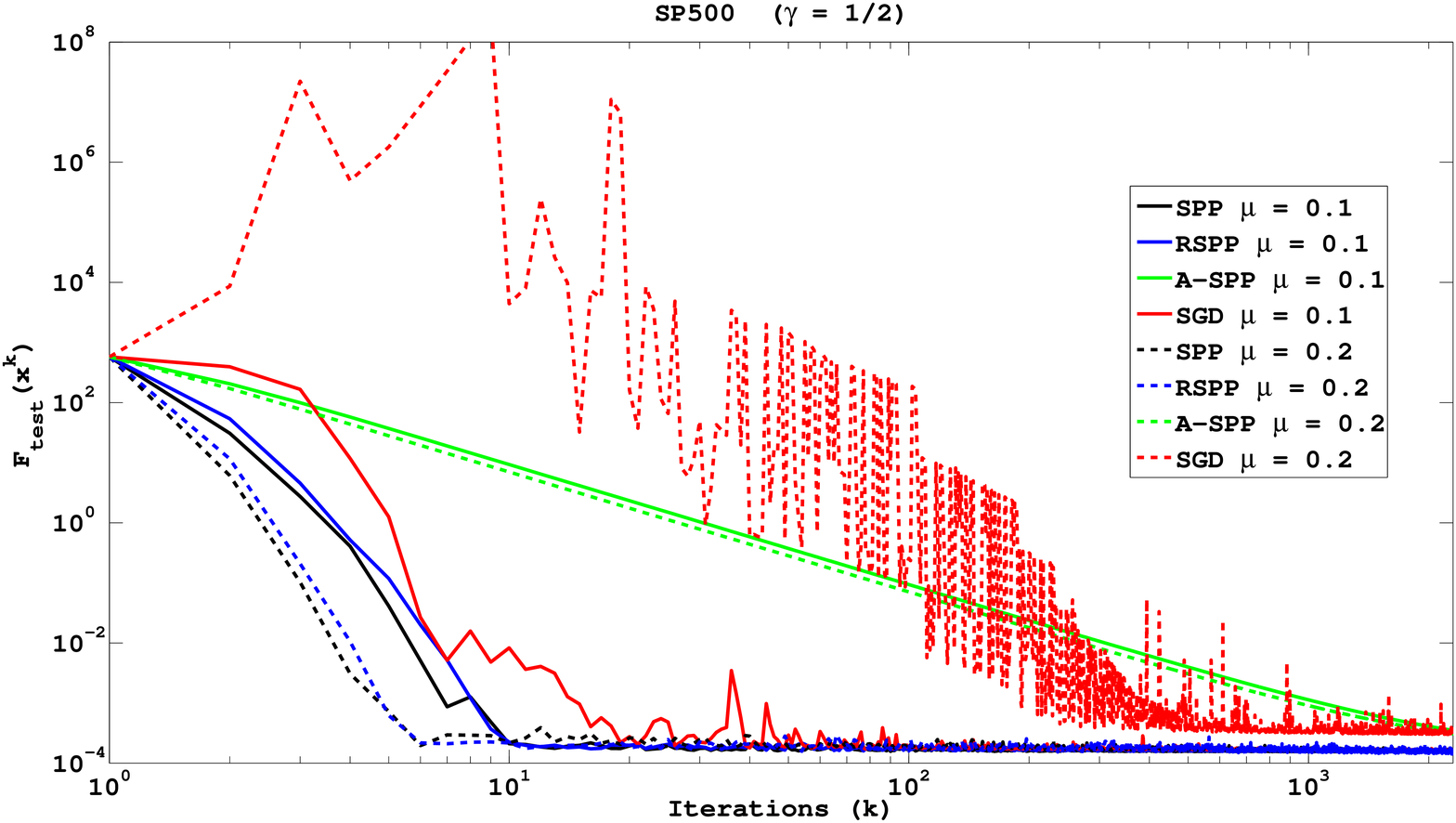}
\includegraphics[width=6.5cm,height=5cm]{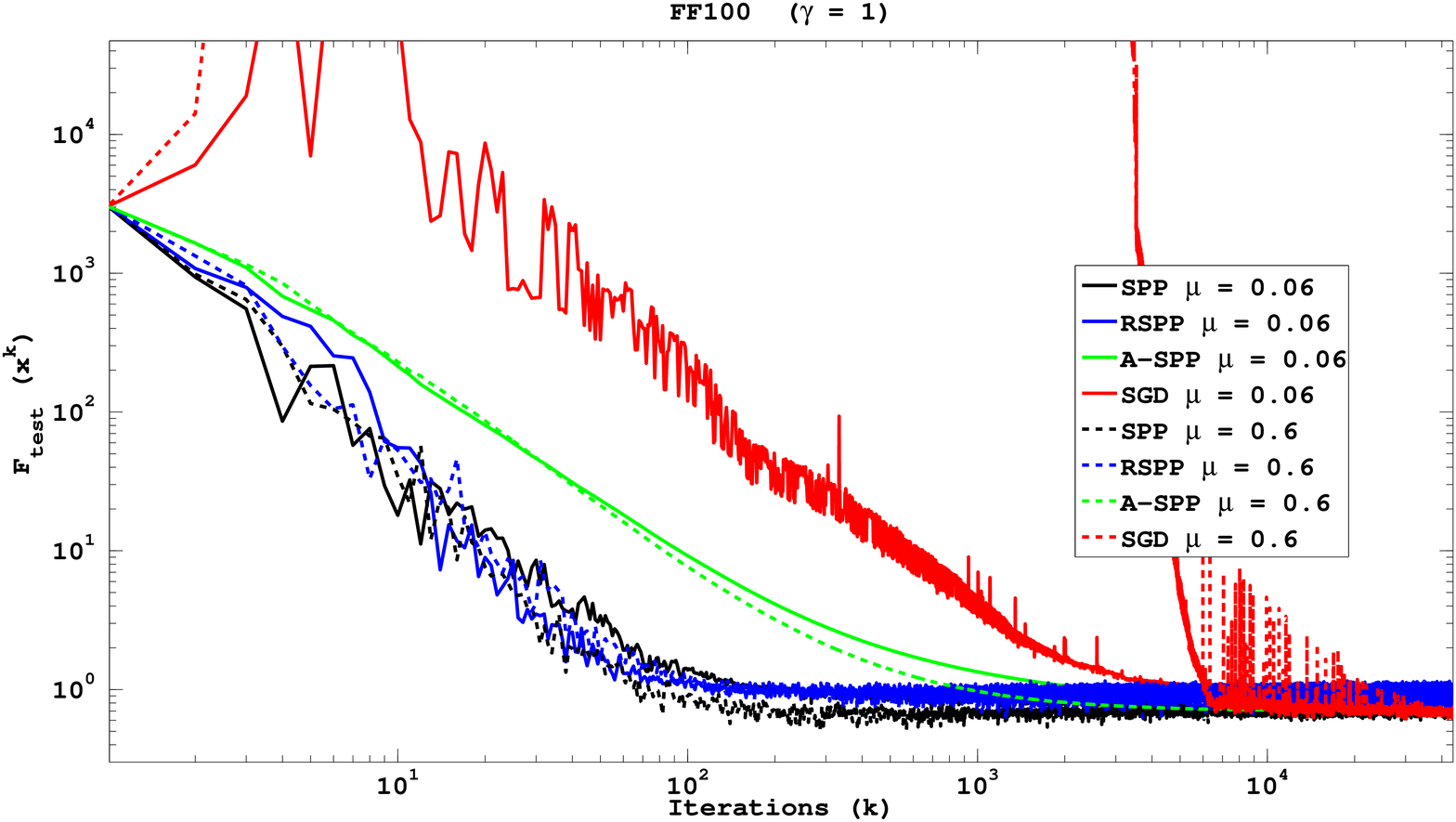}
\includegraphics[width=6.5cm,height=5cm]{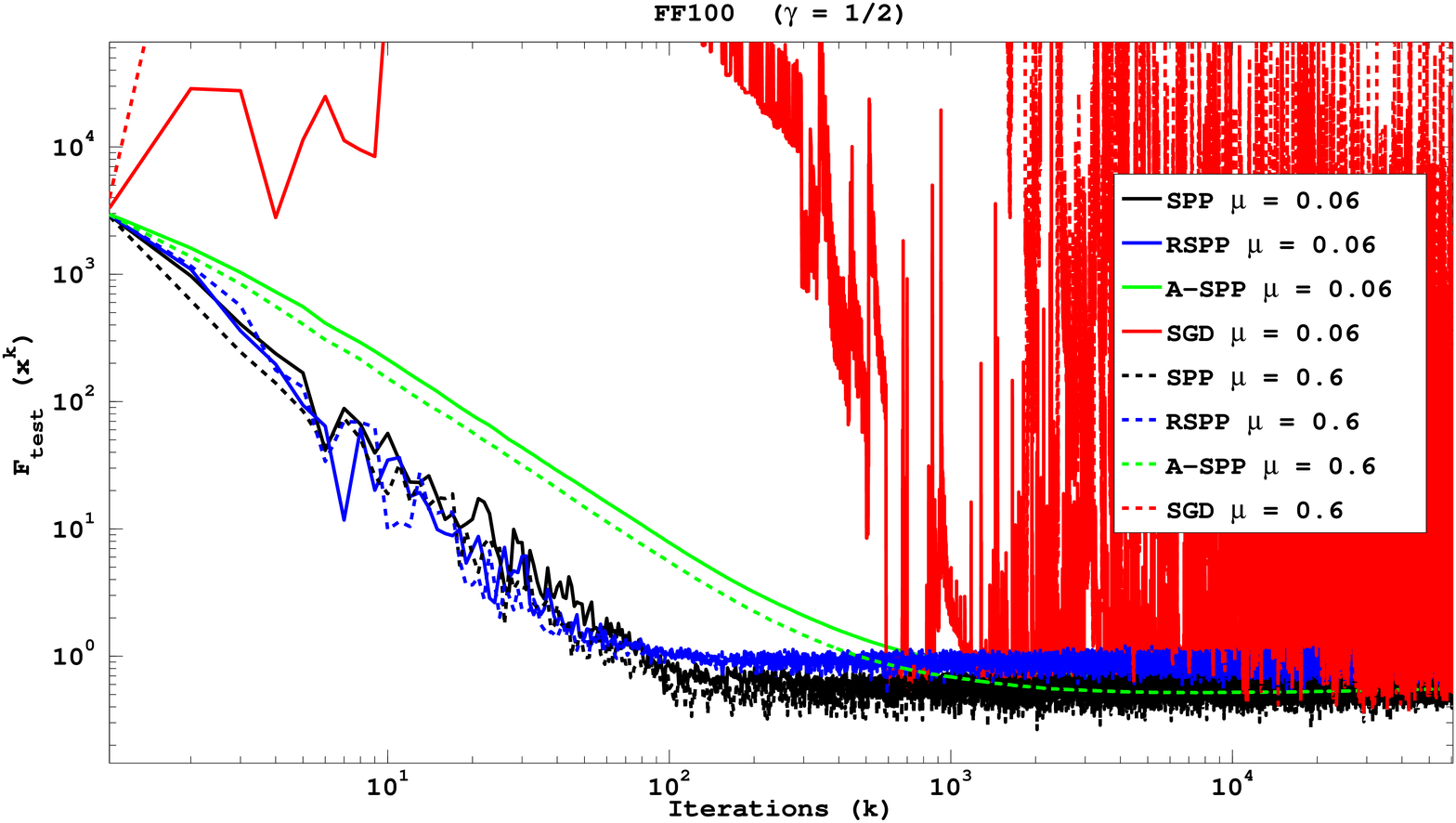}
\end{center}
\caption{Performance on real data of SPP, A-SPP, RSPP and SGD
schemes for several values of the initial stepsize $\mu_0$  and for
two values of the exponent $\gamma=1/2$ (left) and $\gamma=1$
(right): dataset SP500 (top), dataset FF100 (bottom).}
\end{figure}


\section{Acknowledgments}
The research leading to these results has received funding from the
Romanian National Authority for Scientific Research (UEFISCDI),
PNII-RU-TE 2014,  project MoCOBiDS, no. 176/01.10.2015.


\bibliography{StochProx_JMLR}
\end{document}